\documentclass[12pt,reqno,twoside]{amsart}%

\usepackage{amsmath,amsthm,amscd,amsthm,upref,indentfirst}
\usepackage{amsfonts,mathrsfs}

\usepackage{amssymb,amsbsy,bm}
\usepackage{graphicx}
\usepackage{times}
\usepackage{amsaddr}
\usepackage{soul}

\usepackage[us,long,24hr]{datetime}

\usepackage{mathabx}
\usepackage[usenames,dvipsnames]{color}

\theoremstyle{plain}
\newtheorem{theorem}{Theorem}
\newtheorem{assertion}[theorem]{Assertion}

\newtheorem{proposition}[theorem]{Proposition}

\theoremstyle{definition}
\newtheorem{definition}[theorem]{Definition}

\theoremstyle{remark}
\newtheorem{remark}[theorem]{Remark}

\newtheorem{example}[theorem]{Example}
\numberwithin{equation}{section}
\numberwithin{theorem}{section}

\usepackage{exscale}
\usepackage[scaled=0.86]{helvet}
\renewcommand{\mathfrak}[1]{{\textbf{\upshape #1}}}
\renewcommand{\mathbf}{\bm}
\renewcommand{\emph}[1]{\textrm{{\upshape #1}}}
\usepackage[scr=euler,scrscaled=1.15,bb=txof,bbscaled=1.16,cal=boondox,calscaled=1.15]{mathalfa}
\renewcommand{\mathit}[1]{\mathscr #1}
\renewcommand{\mathtt}[1]{\scalebox{1}{\bfseries \texttt{\upshape #1}}}

\usepackage{float}

\usepackage[justification=centering]{caption}

\numberwithin{equation}{section}
\numberwithin{theorem}{section}

\usepackage[normalem]{ulem}
\usepackage[square,comma]{natbib}

\usepackage{mparhack}

\usepackage{keyval}
\usepackage{totcount}
\newtotcounter{citnum} 
\def\oldbibitem{} \let\oldbibitem=\bibitem
\def\bibitem{\stepcounter{citnum}\oldbibitem}

\usepackage[verbose]{backref}
\backrefsetup{verbose=false}
\renewcommand*{\backref}[1]{}
\renewcommand*{\backrefalt}[4]{[{\tiny%
    \ifcase #1 \textsl{Not cited}%
          \or \textsl{Cited on page}~\textcolor{BrickRed}{#2}%
          \else \textsl{Cited on pages}~\textcolor{BrickRed}{#2}%
    \fi%
    }]}

\usepackage{fancyhdr}
\author{\small\scshape S\lowercase{teven} D\lowercase{uplij}}

\address{
Center for Information Technology (WWU IT),
Universit\"at M\"unster,
R\"ontgenstrasse 7-13\\
D-48149 M\"unster,
Deutschland}
\email{\small \sf douplii@uni-muenster.de;
sduplij@gmail.com;
https://ivv5hpp.uni-muenster.de/u/douplii}
\title{\large\bfseries\scshape
P\lowercase{olyadic analog of} G\lowercase{rothendieck group
}}

\date{\textit{of start} February 19, 2022. \textit{Date}:
\textit{of completion} May 8, 2022. \textit{Corrections}: June 14, 2022. 
\newline
\mbox{}\hskip 1.16em
\textit{Total}:
\total{citnum}
references
}

\renewcommand{\refname}{\textsc{References}}

\let\origsection\section
\renewcommand{\section}[1]{\sectionmark{#1}\origsection{#1}}
\let\origsubsection\subsection
\renewcommand{\subsection}[1]{\subsectionmark{#1}\origsubsection{#1}}

\makeatletter
\renewenvironment{thebibliography}[1]{%
  \@xp\origsection\@xp*\@xp{\refname}%
  \normalfont\footnotesize\labelsep .9em\relax
  \renewcommand\theenumiv{\arabic{enumiv}}\let\p@enumiv\@empty
  \vspace*{-5pt}
  \list{\@biblabel{\theenumiv}}{\settowidth\labelwidth{\@biblabel{#1}}%
    \leftmargin\labelwidth \advance\leftmargin\labelsep
    \usecounter{enumiv}}%
  \sloppy \clubpenalty\@M \widowpenalty\clubpenalty
  \sfcode`\.=\@m
}{%
  \def\@noitemerr{\@latex@warning{Empty `thebibliography' environment}}%
  \endlist
}
\makeatother

\subjclass[2010]{16E20, 16T25, 17A42, 18F30, 19A99, 20B30, 20F36, 20M17, 20N15}
\keywords{K-theory, completion group, class group, direct product, polyadic semigroup, arity, polyadic group}

\bibdata{sde.bib,quant.bib,sinbooks-eng.bib,litold.bib,nary.bib,medial.bib,braiding.bib}
\bibstyle{chicago1cc}

\begin{document}
\mbox{}

\mbox{}

\begin{abstract}


\noindent We generalize the Grothendieck construction of the completion group
for a monoid (being the starting point of the algebraic $K$-theory) to the
polyadic case, when an initial semigroup is $m$-ary and the corresponding
final class group $K_{0}$ can be $n$-ary. As opposed to the binary case: 1)
there can be different polyadic direct products which can be built from one
polyadic semigroup; 2) the final arity $n$ of the class groups can be
different from the arity $m$ of initial semigroup; 3) commutative initial
$m$-ary semigroups can lead to noncommutative class $n$-ary groups; 4) the
identity is not necessary for initial $m$-ary semigroup to obtain the class
$n$-ary group, which in its turn can contain no identity at all. The presented
numerical examples show that the properties of the polyadic completion groups
are considerably nontrivial and have more complicated structure than in the
binary case.

\end{abstract}
\maketitle

\thispagestyle{empty}

\mbox{}
\tableofcontents
\newpage

\pagestyle{fancy}

\addtolength{\footskip}{15pt}

\renewcommand{\sectionmark}[1]{%
\markboth{
{ \scshape #1}}{}}

\renewcommand{\subsectionmark}[1]{%
\markright{
\mbox{\;}\\[5pt]
\textmd{#1}}{}}

\fancyhead{}
\fancyhead[EL,OR]{\leftmark}
\fancyhead[ER,OL]{\rightmark}
\fancyfoot[C]{\scshape -- \textcolor{BrickRed}{\thepage} --}
\fancyfoot[R]{{\small {\it \today } \ at \currenttime}}

\renewcommand\headrulewidth{0.5pt}
\fancypagestyle {plain1}{ %
\fancyhf{}
\renewcommand {\headrulewidth }{0pt}
\renewcommand {\footrulewidth }{0pt}
}

\fancypagestyle{plain}{ %
\fancyhf{}
\fancyhead[C]{\scshape S\lowercase{teven} D\lowercase{uplij} \hskip 0.7cm \MakeUppercase{Polyadic Hopf algebras and quantum groups}}
\fancyfoot[C]{\scshape - \thepage  -}
\renewcommand {\headrulewidth }{0pt}
\renewcommand {\footrulewidth }{0pt}
}

\fancypagestyle{fancyref}{ %
\fancyhf{} 
\fancyhead[C]{\scshape R\lowercase{eferences} }
\fancyfoot[C]{\scshape -- \textcolor{BrickRed}{\thepage} --}
\renewcommand {\headrulewidth }{0.5pt}
\renewcommand {\footrulewidth }{0pt}
}

\fancypagestyle{emptyf}{
\fancyhead{}
\fancyfoot[C]{\scshape -- \textcolor{BrickRed}{\thepage} --}
\renewcommand{\headrulewidth}{0pt}
}
\mbox{}
\vskip 3.5cm
\thispagestyle{emptyf}

\section{\textsc{Introduction}}

The Grothendieck construction of the completion group for a monoid is a ground
and starting point of the algebraic $K$-theory (see, e.g.
\cite{karoubi,rosenberg}). Here we generalize this construction to the
polyadic case, when an initial monoid (or semigroup) and a final class group
are $m$-ary and $n$-ary, correspondingly, and we denote such a class polyadic
group $\mathsf{K}_{0}^{\left(  m,n\right)  }$. As concrete examples, we
provide several computations for $\mathsf{K}_{0}^{\left(  m,n\right)  }$,
including the case, when the arities do not coincide $m\neq n$.

\section{\textsc{Preliminaries}}

We introduce here briefly the usual notation, for details see \cite{dup2018a}.
For a non-empty (underlying) set $G$ the $n$\textit{-tuple} (or
\textit{polyad} \cite{pos}) of elements is denoted by $\left(  g_{1}%
,\ldots,g_{n}\right)  $, $g_{i}\in G$, $i=1,\ldots,n$, and the Cartesian
product is denoted by $G^{\times n}\equiv\overset{n}{\overbrace{G\times
\ldots\times G}}$ and consists of all such $n$-tuples. For all elements equal
to $g\in G$, we denote $n$-tuple (polyad) by a power $\left(  g^{n}\right)  $.
To avoid unneeded indices we denote with one bold letter $\left(
\mathbf{g}\right)  $ a polyad for which the number of elements in the
$n$-tuple is clear from the context, and sometimes we will write $\left(
\mathbf{g}^{\left(  n\right)  }\right)  $. On the Cartesian product $G^{\times
n}$ we define a polyadic (or $n$-ary) operation $\mu^{\left(  n\right)
}:G^{\times n}\rightarrow G$ such that $\mu^{\left(  n\right)  }\left[
\mathbf{g}\right]  \mapsto h$, where $h\in G$. The operations with $n=1,2,3$
are called \textit{unary, binary and ternary}.

Recall the definitions of some algebraic structures and their special elements
(in the notation of \cite{dup2018a}). A (one-set) \textit{polyadic algebraic
structure} $\mathcal{G}$ is a set $G$ closed with respect to polyadic
operations. In the case of one $n$-ary operation $\mu^{\left(  n\right)
}:G^{\times n}\rightarrow G$, it is called \textit{polyadic multiplication}
(or $n$\textit{-ary multiplication}). A one-set $n$\textit{-ary algebraic
structure} $\mathcal{M}^{\left(  n\right)  }=\left\langle G\mid\mu^{\left(
n\right)  }\right\rangle $ or \textit{polyadic magma} ($n$-\textit{ary magma)
}is a set $G$ closed with respect to one $n$-ary operation $\mu^{\left(
n\right)  }$ and without any other additional structure. In the binary case
$\mathcal{M}^{\left(  2\right)  }$ was also called a groupoid by Hausmann and
Ore \cite{hau/ore} (and \cite{cli/pre1}). Since the term \textquotedblleft
groupoid\textquotedblright\ was widely used in category theory for a different
construction, the so-called Brandt groupoid \cite{bra1,bruck}, Bourbaki
\cite{bourbaki98} later introduced the term \textquotedblleft
magma\textquotedblright.

Denote the number of iterating multiplications by $\ell_{\mu}$, and call the
resulting composition an \textit{iterated product} $\left(  \mu^{\left(
n\right)  }\right)  ^{\circ\ell_{\mu}}$, such that%
\begin{equation}
\mu^{\prime\left(  n^{\prime}\right)  }=\left(  \mu^{\left(  n\right)
}\right)  ^{\circ\ell_{\mu}}\overset{def}{=}\overset{\ell_{\mu}}%
{\overbrace{\mu^{\left(  n\right)  }\circ\left(  \mu^{\left(  n\right)  }%
\circ\ldots\left(  \mu^{\left(  n\right)  }\times\operatorname*{id}%
\nolimits^{\times\left(  n-1\right)  }\right)  \ldots\times\operatorname*{id}%
\nolimits^{\times\left(  n-1\right)  }\right)  }}, \label{ck-mn}%
\end{equation}
where the arities are connected by%
\begin{equation}
n^{\prime}=n_{iter}=\ell_{\mu}\left(  n-1\right)  +1, \label{ck-n}%
\end{equation}
which gives the length of a iterated polyad $\left(  \mathbf{g}\right)  $ in
our notation $\left(  \mu^{\left(  n\right)  }\right)  ^{\circ\ell_{\mu}%
}\left[  \mathbf{g}\right]  $.

A \textit{polyadic zero} of a polyadic algebraic structure $\mathcal{G}%
^{\left(  n\right)  }\left\langle G\mid\mu^{\left(  n\right)  }\right\rangle $
is a distinguished element $z\in G$ (and the corresponding $0$-ary operation
$\mu_{z}^{\left(  0\right)  }$) such that for any $\left(  n-1\right)  $-tuple
(polyad) $\mathbf{g}^{\left(  n-1\right)  }\mathbf{\in}G^{\times\left(
n-1\right)  }$ we have%
\begin{equation}
\mu^{\left(  n\right)  }\left[  \mathbf{g}^{\left(  n-1\right)  },z\right]
=z, \label{ck-z}%
\end{equation}
where $z$ can be on any place in the l.h.s. of (\ref{ck-z}). If its place is
not fixed it can be a single zero. As in the binary case, an analog of
positive powers of an element \cite{pos} should coincide with the number of
multiplications $\ell_{\mu}$ in the iteration (\ref{ck-mn}).

A (positive) \textit{polyadic power} of an element is%
\begin{equation}
g^{\left\langle \ell_{\mu}\right\rangle }=\left(  \mu^{\left(  n\right)
}\right)  ^{\circ\ell_{\mu}}\left[  g^{\ell_{\mu}\left(  n-1\right)
+1}\right]  . \label{ck-pp}%
\end{equation}
We define associativity as the invariance of the composition of two $n$-ary
multiplications. An element of a polyadic algebraic structure $g$ is called
$\ell_{\mu}$-\textit{nilpotent} (or simply \textit{nilpotent} for $\ell_{\mu
}=1$), if there exist $\ell_{\mu}$ such that%
\begin{equation}
g^{\left\langle \ell_{\mu}\right\rangle }=z. \label{ck-mz}%
\end{equation}
A \textit{polyadic (}$n$-\textit{ary) identity} (or neutral element) of a
polyadic algebraic structure is a distinguished element $e$ (and the
corresponding $0$-ary operation $\mu_{e}^{\left(  0\right)  }$) such that for
any element $g\in G$ we have%
\begin{equation}
\mu^{\left(  n\right)  }\left[  g,e^{n-1}\right]  =g, \label{ck-e}%
\end{equation}
where $g$ can be on any place in the l.h.s. of (\ref{ck-e}).

In polyadic algebraic structures, there exist \textit{neutral polyads}
$\mathbf{n}\in G^{\times\left(  n-1\right)  }$ satisfying%
\begin{equation}
\mu^{\left(  n\right)  }\left[  g,\mathbf{n}\right]  =g, \label{ck-mng}%
\end{equation}
where $g$ can be on any of $n$ places in the l.h.s. of (\ref{ck-mng}).
Obviously, the sequence of polyadic identities $e^{n-1}$ is a neutral polyad
(\ref{ck-e}).

A one-set polyadic algebraic structure $\left\langle G\mid\mu^{\left(
n\right)  }\right\rangle $ is called \textit{totally associative}, if%
\begin{equation}
\left(  \mu^{\left(  n\right)  }\right)  ^{\circ2}\left[  \mathbf{g}%
,\mathbf{h},\mathbf{u}\right]  =\mu^{\left(  n\right)  }\left[  \mathbf{g}%
,\mu^{\left(  n\right)  }\left[  \mathbf{h}\right]  ,\mathbf{u}\right]
=invariant, \label{ck-ghu}%
\end{equation}
with respect to placement of the internal multiplication $\mu^{\left(
n\right)  }\left[  \mathbf{h}\right]  $ in r.h.s. on any of $n$ places, with a
fixed order of elements in the any fixed polyad of $\left(  2n-1\right)  $
elements $\mathbf{t}^{\left(  2n-1\right)  }=\left(  \mathbf{g},\mathbf{h}%
,\mathbf{u}\right)  \in G^{\times\left(  2n-1\right)  }$.

A \textit{polyadic semigroup} $\mathcal{S}^{\left(  n\right)  }$ is a one-set
$S$ one-operation $\mu^{\left(  n\right)  }$ algebraic structure in which the
$n$-ary multiplication is associative, $\mathcal{S}^{\left(  n\right)
}=\left\langle S\mid\mu^{\left(  n\right)  }\mid\text{associativity
(\ref{ck-ghu})}\right\rangle $. A polyadic algebraic structure $\mathcal{G}%
^{\left(  n\right)  }=\left\langle G\mid\mu^{\left(  n\right)  }\right\rangle
$ is $\sigma$-\textit{commutative}, if $\mu^{\left(  n\right)  }=\mu^{\left(
n\right)  }\circ\sigma$, or%
\begin{equation}
\mu^{\left(  n\right)  }\left[  \mathbf{g}\right]  =\mu^{\left(  n\right)
}\left[  \sigma\circ\mathbf{g}\right]  ,\ \ \ \mathbf{g}\in G^{\times n},
\label{ck-ms}%
\end{equation}
where $\sigma\circ\mathbf{g}=\left(  g_{\sigma\left(  1\right)  }%
,\ldots,g_{\sigma\left(  n\right)  }\right)  $ is a permutated polyad and
$\sigma$ is a fixed element of $S_{n}$, the permutation group on $n$ elements.
If (\ref{ck-ms}) holds for all $\sigma\in S_{n}$, then a polyadic algebraic
structure is \textit{commutative}. A special type of the $\sigma
$-commutativity%
\begin{equation}
\mu^{\left(  n\right)  }\left[  g,\mathbf{t}^{\left(  n-2\right)  },h\right]
=\mu^{\left(  n\right)  }\left[  h,\mathbf{t}^{\left(  n-2\right)  },g\right]
, \label{ck-mth}%
\end{equation}
where $\mathbf{t}^{\left(  n-2\right)  }\in G^{\times\left(  n-2\right)  }$ is
any fixed $\left(  n-2\right)  $-polyad, is called \textit{semicommutativity}.
If an $n$-ary semigroup $\mathcal{S}^{\left(  n\right)  }$ is iterated from a
commutative binary semigroup with identity, then $\mathcal{S}^{\left(
n\right)  }$ is semicommutative. A polyadic algebraic structure is called
(uniquely) $i$-\textit{solvable}, if for all polyads $\mathbf{t}$,
$\mathbf{u}$ and element $h$, one can (uniquely) resolve the equation (with
respect to $h$) for the fundamental operation%
\begin{equation}
\mu^{\left(  n\right)  }\left[  \mathbf{u},h,\mathbf{t}\right]  =g
\label{ck-mug}%
\end{equation}
where $h$ can be on any place, and $\mathbf{u},\mathbf{t}$ are polyads of the
needed length.

A polyadic algebraic structure which is uniquely $i$-solvable for all places
$i=1,\ldots,n$ is called a $n$-\textit{ary }(or \textit{polyadic})\textit{
quasigroup }$\mathcal{Q}^{\left(  n\right)  }=\left\langle Q\mid\mu^{\left(
n\right)  }\mid\text{solvability}\right\rangle $. An associative polyadic
quasigroup is called a $n$-\textit{ary} (or \textit{polyadic})\textit{ group}.
In an $n$-ary group $\mathcal{G}^{\left(  n\right)  }=\left\langle G\mid
\mu^{\left(  n\right)  }\right\rangle $ the only solution of (\ref{ck-mug}) is
called a \textit{querelement} of $g$ and denoted by $\bar{g}$ \cite{dor3},
such that%
\begin{equation}
\mu^{\left(  n\right)  }\left[  \mathbf{h},\bar{g}\right]  =g,\ \ \ g,\bar
{g}\in G, \label{ck-mgg}%
\end{equation}
where $\bar{g}$ can be on any place. Any idempotent $g$ coincides with its
querelement $\bar{g}=g$. The unique solvability relation (\ref{ck-mgg}) in a
$n$-ary group can be treated as a definition of the unary (multiplicative)
\textit{queroperation}%
\begin{equation}
\bar{\mu}^{\left(  1\right)  }\left[  g\right]  =\bar{g}. \label{ck-m1g}%
\end{equation}
We observe from (\ref{ck-mgg}) and (\ref{ck-mng}) that the polyad%
\begin{equation}
\mathbf{n}_{g}=\left(  g^{n-2}\bar{g}\right)  \label{ck-ng}%
\end{equation}
is neutral for any element of a polyadic group, where $\bar{g}$ can be on any
place. If this $i$-th place is important, then we write $\mathbf{n}_{g;i}$. In
a polyadic group the \textit{D\"{o}rnte relations} \cite{dor3}%
\begin{equation}
\mu^{\left(  n\right)  }\left[  g,\mathbf{n}_{h;i}\right]  =\mu^{\left(
n\right)  }\left[  \mathbf{n}_{h;j},g\right]  =g \label{ck-mgnn}%
\end{equation}
hold true for any allowable $i,j$. In the case of a binary group the relations
(\ref{ck-mgnn}) become $g\cdot h\cdot h^{-1}=h\cdot h^{-1}\cdot g=g$.

Using the queroperation (\ref{ck-m1g}) one can give a \textit{diagrammatic
definition} of a polyadic group \cite{gle/gla}: an $n$-\textit{ary group} is a
one-set algebraic structure (universal algebra)%
\begin{equation}
\mathcal{G}^{\left(  n\right)  }=\left\langle G\mid\mu^{\left(  n\right)
},\bar{\mu}^{\left(  1\right)  }\mid\text{associativity (\ref{ck-ghu}),
D\"{o}rnte relations (\ref{ck-mgnn}) }\right\rangle , \label{ck-diam5}%
\end{equation}
where $\mu^{\left(  n\right)  }$ is a $n$-ary associative multiplication and
$\bar{\mu}^{\left(  1\right)  }$ is the queroperation (\ref{ck-m1g}).

\section{\textsc{Grothendieck group of commutative monoid}}

First, we describe the standard Grothendieck construction (see, e.g.
\cite{karoubi,rosenberg,weibel}) of a commutative group from a commutative
semigroup with identity (monoid). We will use multiplicative notation, which
will allow us to provide a straightforward \textquotedblleft
polyadization\textquotedblright\ according to the arity invariance principle
\cite{dup2021b}.

Let us have a (binary, arity $m=2$) commutative monoid $\mathcal{S}%
\equiv\mathcal{S}^{\left(  2\right)  }=\left\langle S\mid\mu=\mu^{\left(
2\right)  }\cong\left(  \cdot\right)  \mid assoc\right\rangle $, where $S$ is
the underlying set, and $\mu=\mu^{\left(  2\right)  }:S\times S\rightarrow S$
is the (associative) multiplication in $\mathcal{S}$. The Cartesian product of
two underlying sets $S^{\prime}=S\times S$ can be endowed with the
componentwise multiplication $\left(  a_{1},b_{1}\right)  \bullet^{\prime
}\left(  a_{2},b_{2}\right)  =\left(  a_{1}\cdot a_{2},b_{1}\cdot
b_{2}\right)  $, $a_{i},b_{i}\in S$, to define the binary direct product
$\mathcal{S}^{\prime}=\mathcal{S}^{\prime\left(  2\right)  }=\left\langle
S^{\prime}\mid\mathbf{\mu}^{\prime}=\mathbf{\mu}^{\prime\left(  2\right)
}\equiv\left(  \bullet^{\prime}\right)  \right\rangle $ (which coincides with
the direct sum, because of the finite number of factors in the product).

For convenience and conciseness, we introduce the doubles%
\begin{equation}
\mathbf{S}=\left(
\begin{array}
[c]{c}%
a\\
b
\end{array}
\right)  \in S\times S \label{ck-sab}%
\end{equation}
and use vector-like notation for the multiplication (being the Kronecker
product of the doubles)%
\begin{equation}
\mathbf{S}_{1}\bullet^{\prime}\mathbf{S}_{2}=\left(
\begin{array}
[c]{c}%
a_{1}\\
b_{1}%
\end{array}
\right)  \bullet^{\prime}\left(
\begin{array}
[c]{c}%
a_{2}\\
b_{2}%
\end{array}
\right)  =\left(
\begin{array}
[c]{c}%
a_{1}\cdot a_{2}\\
b_{1}\cdot b_{2}%
\end{array}
\right)  ,\ \ \mathbf{S}_{i}\in S^{\prime},\ \ a_{i},b_{i}\in S, \label{ck-gg}%
\end{equation}
or in \textquotedblleft polyadic\textquotedblright\ notation%
\begin{equation}
\mathbf{\mu}^{\prime}\left[  \mathbf{S}_{1},\mathbf{S}_{2}\right]  =\left(
\begin{array}
[c]{c}%
\mu\left[  a_{1},a_{2}\right] \\[5pt]%
\mu\left[  b_{1},b_{2}\right]
\end{array}
\right)  . \label{ck-mg12}%
\end{equation}

The associativity of the direct product $\mathbf{\mu}^{\prime}$ follows
immediately from that of $\mu$, because of the componentwise multiplication in
(\ref{ck-mg12}). Since $\mathcal{S}$ is a monoid with the neutral element
(identity) $e\in S$, satisfying $\mu\left[  e,a\right]  =\mu\left[
a,e\right]  =a$, $a\in S$, then the identity of the direct product
$\mathcal{S}^{\prime}$ is the double%
\begin{equation}
\mathbf{E}=\left(
\begin{array}
[c]{c}%
e\\
e
\end{array}
\right)  , \label{ck-ee0}%
\end{equation}
such that%
\begin{equation}
\mathbf{\mu}^{\prime}\left[  \mathbf{E},\mathbf{G}\right]  =\mathbf{\mu
}^{\prime}\left[  \mathbf{G},\mathbf{E}\right]  =\mathbf{G}\in S\times S.
\label{ck-me}%
\end{equation}
Therefore, $\mathcal{S}^{\prime}=\mathcal{S}\times\mathcal{S}$ is a
commutative monoid, as is $\mathcal{S}$.

Another associative direct product $\mathcal{S}^{\prime\prime}=\mathcal{S}%
^{\prime\prime\left(  2\right)  }=\left\langle S^{\prime}\mid\mathbf{\mu
}^{\prime\prime}=\mathbf{\mu}_{2}^{\prime\prime}\equiv\left(  \bullet
^{\prime\prime}\right)  \right\rangle $ can be obtained using the
\textquotedblleft twisted\textquotedblright\ multiplication of the doubles
$\mathbf{\mu}^{\prime\prime}$ defined by%
\begin{equation}
\mathbf{S}_{1}\bullet^{\prime\prime}\mathbf{S}_{2}=\left(
\begin{array}
[c]{c}%
a_{1}\\
b_{1}%
\end{array}
\right)  \bullet^{\prime\prime}\left(
\begin{array}
[c]{c}%
a_{2}\\
b_{2}%
\end{array}
\right)  =\left(
\begin{array}
[c]{c}%
a_{1}\cdot b_{2}\\
a_{2}\cdot b_{1}%
\end{array}
\right)  , \label{ck-ggt}%
\end{equation}
or%
\begin{equation}
\mathbf{\mu}^{\prime\prime}\left[  \mathbf{S}_{1},\mathbf{S}_{2}\right]
=\left(
\begin{array}
[c]{c}%
\mu_{2}\left[  a_{1},b_{2}\right] \\[5pt]%
\mu_{2}\left[  a_{2},b_{1}\right]
\end{array}
\right)  . \label{ck-mg12t}%
\end{equation}

The neutral element (identity) $\mathbf{E}$ in the \textquotedblleft
twisted\textquotedblright\ direct product $\mathcal{S}^{\prime\prime}$
coincides with (\ref{ck-ee0}), and therefore $\mathcal{S}^{\prime\prime
}=\mathcal{S}\times\mathcal{S}$ is a commutative monoid as well.

The question arises: how to construct a (binary) group corresponding to the
binary monoid $\mathcal{S}=\mathcal{S}_{2}$ which would reflect its
substantial and important properties? The answer was provided by Grothendieck:
to consider the equivalence relations and corresponding classes in the direct
product $\mathcal{S}\times\mathcal{S}$. Because, indeed on classes one can
define the inverse elements which are needed to build a group (in addition to
associativity and existence of neutral elements which are sufficient for
monoids). Here we briefly reproduce the construction of the Grothendieck group
corresponding to the commutative monoid $\mathcal{S}$ (sometimes this is
called the symmetrization of $\mathcal{S}$ \cite{karoubi} or the group
completion of $\mathcal{S}$ \cite{rosenberg,weibel}) in multiplicative
notation, which will allow us to provide its \textquotedblleft
polyadization\textquotedblright\ in a straightforward way.

Let us consider two kinds of equivalence relations on the direct product
$\mathcal{S}^{\prime}=\mathcal{S}\times\mathcal{S}$. The first one $\left(
\sim_{1}\right)  $ is reminiscent of \textquotedblleft gauge
invariance\textquotedblright\ (in physical language), because it identifies
the doubles with equal \textquotedblleft shifts\textquotedblright, such that%
\begin{equation}
\left(
\begin{array}
[c]{c}%
a_{1}\\
b_{1}%
\end{array}
\right)  \sim_{1}\left(
\begin{array}
[c]{c}%
a_{2}\\
b_{2}%
\end{array}
\right)  \Longleftrightarrow\exists_{x,y\in S}\ \left(
\begin{array}
[c]{c}%
\mu\left[  a_{1},x\right] \\[5pt]%
\mu\left[  b_{1},x\right]
\end{array}
\right)  =\left(
\begin{array}
[c]{c}%
\mu\left[  a_{2},y\right] \\[5pt]%
\mu\left[  b_{2},y\right]
\end{array}
\right)  ,\ \ a_{i},b_{i}\in S, \label{ck-m2a}%
\end{equation}
and we call it the \textit{\textquotedblleft gauge\textquotedblright\ shifts}.
The second equivalence relation $\left(  \sim_{2}\right)  $ uses only one
\textquotedblleft shift\textquotedblright\ as follows%
\begin{equation}
\left(
\begin{array}
[c]{c}%
a_{1}\\
b_{1}%
\end{array}
\right)  \sim_{2}\left(
\begin{array}
[c]{c}%
a_{2}\\
b_{2}%
\end{array}
\right)  \Longleftrightarrow\exists_{z\in S}\ \left(  \mu\right)  ^{\circ
2}\left[  a_{1},b_{2},z\right]  =\left(  \mu\right)  ^{\circ2}\left[
a_{2},b_{1},z\right]  ,\ \ a_{i},b_{i}\in S, \label{ck-m2b}%
\end{equation}
so we call this the \textit{\textquotedblleft twisted\textquotedblright%
\ shift} (which was used originally by Grothendieck).

\begin{assertion}
\label{ck-as-eq}Two equivalence relations above coincide $\sim_{1}=\sim
_{2}\equiv\sim$.
\end{assertion}

\begin{proof}
Let (\ref{ck-m2b}) holds, then putting $x=\mu\left[  z,b_{2}\right]  $,
$y=\mu\left[  a_{1},z\right]  $, we obtain (\ref{ck-m2a}). Conversely, from
(\ref{ck-m2a}) with%
\begin{equation}
z=\mu\left[  x,y\right]  \label{ck-zxy}%
\end{equation}
it follows that%
\begin{equation}
\left(  \mu\right)  ^{\circ2}\left[  a_{1},b_{2},z\right]  =\mu\left[
\mu\left[  a_{1},x\right]  ,\mu\left[  y,b_{2}\right]  \right]  =\mu\left[
\mu\left[  y,b_{1}\right]  ,\mu\left[  a_{2},x\right]  \right]  =\left(
\mu\right)  ^{\circ2}\left[  a_{2},b_{1},z\right]  .
\end{equation}

\end{proof}

The group completion of the commutative monoid $\mathcal{S}$ is defined as the
(binary $n=2$) group $\widetilde{\mathcal{G}}^{\left(  2\right)  }$ (of
isomorphism classes) being the factorization of $\mathcal{S}\times\mathcal{S}$
by the equivalence relation $\left(  \sim\right)  $ given in (\ref{ck-m2a}) or
(\ref{ck-m2b}). It is called the Grothendieck group, and is usually denoted by
$\mathsf{K}_{0}\left(  \mathcal{S}\right)  =\widetilde{\mathcal{G}}^{\left(
2\right)  }$ (for this concrete case of monoid \cite{rosenberg,weibel})%
\begin{equation}
\mathsf{K}_{0}\left(  \mathcal{S}\right)  =\mathcal{S}\times\mathcal{S}%
\diagup\sim. \label{ck-gs}%
\end{equation}

The representatives, \textquotedblleft observables\textquotedblright\ (in
physical language), are \textquotedblleft gauge invariant\textquotedblright%
\ doubles (\ref{ck-m2a})%
\begin{equation}
\widetilde{\mathfrak{G}}=\left[
\begin{array}
[c]{c}%
a\\
b
\end{array}
\right]  . \label{ck-dg}%
\end{equation}

The multiplication of the representatives $\widetilde{\mathfrak{G}}$ in
$\mathsf{K}_{0}\left(  \mathcal{S}\right)  $ inherits the product of doubles
in $\mathcal{S}\times\mathcal{S}$ (\ref{ck-mg12})%
\begin{equation}
\widetilde{\mathbf{\mu}}\left[  \widetilde{\mathfrak{G}}_{1},\widetilde
{\mathfrak{G}}_{2}\right]  =\left[
\begin{array}
[c]{c}%
\mu\left[  a_{1},a_{2}\right] \\[5pt]%
\mu\left[  b_{1},b_{2}\right]
\end{array}
\right]  , \label{ck-mw}%
\end{equation}
where the elements of the resulting binary factor group $\mathsf{K}_{0}\left(
\mathcal{S}\right)  =\widetilde{\mathcal{G}}^{\left(  2\right)  }%
=\widetilde{\mathcal{G}}=\left\langle \left\{  \widetilde{\mathfrak{G}%
}\right\}  \mid\widetilde{\mathbf{\mu}}\right\rangle $ (the representative
doubles $\widetilde{\mathfrak{G}}$ (\ref{ck-dg}) and the binary operation
$\widetilde{\mathbf{\mu}}=\widetilde{\mathbf{\mu}}^{\left(  2\right)  }$ on
the classes are marked by waves). The structure of the Grothendieck group for
monoids $\mathsf{K}_{0}\left(  \mathcal{S}\right)  $ was considered as the
staring example in, e.g., \cite{karoubi, rosenberg, weibel}.

The first equivalence relation $\left(  \sim_{1}\right)  $ gives the form of
the neutral element in $\mathsf{K}_{0}\left(  \mathcal{S}\right)  $%
\begin{equation}
\widetilde{\mathfrak{E}}=\left[
\begin{array}
[c]{c}%
e\\
e
\end{array}
\right]  \sim\left[
\begin{array}
[c]{c}%
a\\
a
\end{array}
\right]  ,\ \ \forall a\in S, \label{ck-gee}%
\end{equation}
where $e\in S$ is the identity of the monoid $\mathcal{S}$. Indeed the shape
(\ref{ck-gee}) of the neutral element and commutativity of the initial monoid
$\mathcal{S}$ allows us to obtain the inverse element in $\mathsf{K}%
_{0}\left(  \mathcal{S}\right)  $ using the multiplication (\ref{ck-mw}) in
the following way
\begin{equation}
\widetilde{\mathbf{\mu}}\left[  \left[
\begin{array}
[c]{c}%
a\\
b
\end{array}
\right]  ,\left[
\begin{array}
[c]{c}%
b\\
a
\end{array}
\right]  \right]  =\left[
\begin{array}
[c]{c}%
\mu\left[  a,b\right] \\[5pt]%
\mu\left[  b,a\right]
\end{array}
\right]  =\left[
\begin{array}
[c]{c}%
\mu\left[  a,b\right] \\[5pt]%
\mu\left[  a,b\right]
\end{array}
\right]  \sim\left[
\begin{array}
[c]{c}%
e\\
e
\end{array}
\right]  , \label{ck-ma0}%
\end{equation}
or%
\begin{equation}
\widetilde{\mathbf{\mu}}\left[  \widetilde{\mathfrak{G}},\widetilde
{\mathfrak{G}}^{-1}\right]  =\widetilde{\mathfrak{E}}, \label{ck-mgge}%
\end{equation}
where the unique inverse is%
\begin{equation}
\widetilde{\mathfrak{G}}^{-1}=\left[
\begin{array}
[c]{c}%
a\\
b
\end{array}
\right]  ^{-1}=\left[
\begin{array}
[c]{c}%
b\\
a
\end{array}
\right]  . \label{ck-g1}%
\end{equation}

Thus, $\widetilde{\mathcal{G}}=\mathsf{K}_{0}\left(  \mathcal{S}\right)  $ is
indeed a (binary) commutative group (of classes) corresponding to the (binary)
monoid $\mathcal{S}=\mathcal{S}_{2}$. Using (\ref{ck-m2a}) and (\ref{ck-gee}),
the homomorphism (of monoids) $\mathbf{\Phi}_{SG}:\mathcal{S}\rightarrow
\mathsf{K}_{0}\left(  \mathcal{S}\right)  $ can be written as%
\begin{equation}
\mathbf{\Phi}_{SG}\left(  a\right)  =\left[
\begin{array}
[c]{c}%
\mu\left[  a,a\right] \\
a
\end{array}
\right]  ,\ \ \ \ \forall a\in S, \label{ck-fsg}%
\end{equation}
where we did not use the identity $e\in S$ (this can be important in the
\textquotedblleft polyadization\textquotedblright\ below). It follows from
(\ref{ck-gee}), that $\mathbf{\Phi}_{SG}$ can be written with the identity in
the form%
\begin{equation}
\mathbf{\Phi}_{SG}\left(  a\right)  =\left[
\begin{array}
[c]{c}%
a\\
e
\end{array}
\right]  ,\ \ \ \ \forall a\in S,\ e\in S. \label{ck-fsg1}%
\end{equation}

Using (\ref{ck-g1}) we observe that the image of $\mathbf{\Phi}_{SG}$ actually
generates a group (being $\mathsf{K}_{0}\left(  \mathcal{S}\right)  $),
because%
\begin{equation}
\left[
\begin{array}
[c]{c}%
a\\
b
\end{array}
\right]  =\widetilde{\mathbf{\mu}}\left[  \mathbf{\Phi}_{SG}\left(  a\right)
,\left(  \mathbf{\Phi}_{SG}\left(  b\right)  \right)  ^{-1}\right]
,\ \ \forall a,b\in S. \label{ck-abm}%
\end{equation}

The universal property can be shown in the following way (see, e.g.,
\cite{rosenberg}). Let us consider any commutative (binary) group
$\widetilde{\mathcal{G}}^{\prime}=\left\langle G^{\prime}\mid\widetilde
{\mathbf{\mu}}^{\prime}\right\rangle $ and the group homomorphism
$\mathbf{\Phi}_{GG^{\prime}}:\widetilde{\mathcal{G}}\rightarrow\widetilde
{\mathcal{G}}^{\prime}$, then there exists the unique homomorphism
$\mathbf{\Phi}_{SG^{\prime}}:\mathcal{S}\rightarrow\widetilde{\mathcal{G}%
}^{\prime}$, such that%
\begin{equation}
\mathbf{\Phi}_{SG^{\prime}}=\mathbf{\Phi}_{GG^{\prime}}\circ\mathbf{\Phi}%
_{SG}. \label{ck-fff}%
\end{equation}

Indeed, we derive, using (\ref{ck-fsg1}) and (\ref{ck-abm})%
\begin{align}
\mathbf{\Phi}_{GG^{\prime}}\left(  \left[
\begin{array}
[c]{c}%
a\\
b
\end{array}
\right]  \right)   &  =\mathbf{\Phi}_{GG^{\prime}}\left(  \widetilde
{\mathbf{\mu}}\left[  \left[
\begin{array}
[c]{c}%
a\\
e
\end{array}
\right]  ,\left(  \left[
\begin{array}
[c]{c}%
b\\
e
\end{array}
\right]  \right)  ^{-1}\right]  \right)  =\widetilde{\mathbf{\mu}}^{\prime
}\left[  \mathbf{\Phi}_{GG^{\prime}}\left(  \left[
\begin{array}
[c]{c}%
a\\
e
\end{array}
\right]  \right)  ,\left(  \mathbf{\Phi}_{GG^{\prime}}\left(  \left[
\begin{array}
[c]{c}%
b\\
e
\end{array}
\right]  \right)  \right)  ^{-1}\right] \nonumber\\
&  =\widetilde{\mathbf{\mu}}\left[  \mathbf{\Phi}_{GG^{\prime}}\left(
\mathbf{\Phi}_{SG}\left(  a\right)  \right)  ,\left(  \mathbf{\Phi
}_{GG^{\prime}}\left(  \mathbf{\Phi}_{SG}\left(  b\right)  \right)  \right)
^{-1}\right]  =\widetilde{\mathbf{\mu}}^{\prime}\left[  \mathbf{\Phi
}_{SG^{\prime}}\left(  a\right)  ,\left(  \mathbf{\Phi}_{SG^{\prime}}\left(
b\right)  \right)  ^{-1}\right]  . \label{ck-ff}%
\end{align}
Conversely, for a given homomorphism $\mathbf{\Phi}_{SG^{\prime}}%
:\mathcal{S}\rightarrow\widetilde{\mathcal{G}}^{\prime}$ the group
homomorphism $\mathbf{\Phi}_{GG^{\prime}}:\widetilde{\mathcal{G}}%
\rightarrow\widetilde{\mathcal{G}}^{\prime}$ is uniquely defined by
(\ref{ck-ff}) with $\mathbf{\Phi}_{SG^{\prime}}=\mathbf{\Phi}_{GG^{\prime}%
}\circ\mathbf{\Phi}_{SG}$ (\ref{ck-fff}).

\begin{example}
The simplest example is the commutative monoid of nonnegative integers
(natural numbers with the zero $\mathbb{N}_{0}$) under addition $\mathcal{S}%
=\left\langle \mathbb{N}_{0}\mid\mu=\left(  +\right)  \right\rangle $, and
$e=0$. The elements of the $\mathsf{K}_{0}\left(  \mathcal{S}\right)
=\left\langle \mathbb{N}_{0}\times\mathbb{N}_{0}\mid\widetilde{\mathbf{\mu}%
}=\left(  \widetilde{+}\right)  \right\rangle $ are doubles of natural numbers
$\left(
\begin{array}
[c]{c}%
n\\
m
\end{array}
\right)  \in\mathbb{N}_{0}\times\mathbb{N}_{0}$, and the neutral double is
$\left(
\begin{array}
[c]{c}%
0\\
0
\end{array}
\right)  $. The equivalence relation (\ref{ck-m2b}) becomes%
\begin{equation}
\left(
\begin{array}
[c]{c}%
n_{1}\\
m_{1}%
\end{array}
\right)  \sim\left(
\begin{array}
[c]{c}%
n_{2}\\
m_{2}%
\end{array}
\right)  \Longleftrightarrow n_{1}+m_{2}=n_{2}+m_{1},\ \ \ n_{i},m_{i}%
\in\mathbb{N}_{0}, \label{ck-nm}%
\end{equation}
because the monoid $\mathcal{S}$ is cancellative. It follows from
(\ref{ck-nm}) by scaling, that in $\mathsf{K}_{0}\left(  \mathcal{S}\right)  $
there exist two minimal representatives $\left[
\begin{array}
[c]{c}%
n\\
0
\end{array}
\right]  $ and $\left[
\begin{array}
[c]{c}%
0\\
m
\end{array}
\right]  $, and from (\ref{ck-ma0}) we have%
\begin{equation}
\left[
\begin{array}
[c]{c}%
n\\
0
\end{array}
\right]  \widetilde{+}\left[
\begin{array}
[c]{c}%
0\\
n
\end{array}
\right]  =\left[
\begin{array}
[c]{c}%
n\\
n
\end{array}
\right]  =\left[
\begin{array}
[c]{c}%
0\\
0
\end{array}
\right]  .
\end{equation}

Therefore, the representatives $\left[
\begin{array}
[c]{c}%
n\\
0
\end{array}
\right]  $ and $\left[
\begin{array}
[c]{c}%
0\\
m
\end{array}
\right]  $ can be treated as positive and negative integers in $\mathbb{Z}$,
such that there exists the homomorphism $\mathsf{K}_{0}\left(  \mathbb{N}%
_{0}\right)  \rightarrow\mathbb{Z}$ defined by%
\begin{equation}
\left[
\begin{array}
[c]{c}%
n\\
m
\end{array}
\right]  \mapsto n-m\in\mathbb{Z}, \label{ck-nmz}%
\end{equation}
which is a bijection due to (\ref{ck-nm}). The universal property
(\ref{ck-ff}) becomes%
\begin{align}
\mathbf{\Phi}_{GG^{\prime}}\left(  \left[
\begin{array}
[c]{c}%
n\\
m
\end{array}
\right]  \right)   &  =\mathbf{\Phi}_{GG^{\prime}}\left(  \left[
\begin{array}
[c]{c}%
n\\
0
\end{array}
\right]  -\left[
\begin{array}
[c]{c}%
m\\
0
\end{array}
\right]  \right)  =\mathbf{\Phi}_{GG^{\prime}}\left(  \left[
\begin{array}
[c]{c}%
n\\
0
\end{array}
\right]  \right)  -\mathbf{\Phi}_{GG^{\prime}}\left(  \left[
\begin{array}
[c]{c}%
m\\
0
\end{array}
\right]  \right) \nonumber\\
&  =\mathbf{\Phi}_{GG^{\prime}}\left(  \mathbf{\Phi}_{SG}\left(  n\right)
\right)  -\mathbf{\Phi}_{GG^{\prime}}\left(  \mathbf{\Phi}_{SG}\left(
m\right)  \right)  =\mathbf{\Phi}_{SG^{\prime}}\left(  n\right)
-\mathbf{\Phi}_{SG^{\prime}}\left(  m\right)  .
\end{align}
Thus, the Grothendieck group of the monoid $\mathbb{N}_{0}$ of natural numbers
with zero is the group of integers $\mathbb{Z}$ under addition $\mathsf{K}%
_{0}\left(  \mathbb{N}_{0}\right)  =\mathbb{Z}$.
\end{example}

\section{$n$\textsc{-ary group completion of }$m$\textsc{-ary semigroup}}

Here we propose the \textquotedblleft polyadization\textquotedblright\ (along
the ideas of \cite{dup2019}) of the group completion concept using the arity
invariance principle \cite{dup2021b}. By considering the polyadic algebraic
structures, the main differences with the binary case will be the following
\cite{dup2022} (and refs therein):

\begin{enumerate}
\item There can exist several associative polyadic direct products (powers)
which can be built from one polyadic semigroup.

\item If initial $m$-ary semigroup is commutative, the polyadic direct product
can be noncommutative.

\item The arity $n$ of the direct product of $m$-ary semigroups can be not
equal to $m$.

\item The neutral element (identity) is not necessary for the $m$-ary
semigroup and its power, because the polyadic analog of the binary inverse
element is the querelement of an $n$-ary group, which is defined without the
usage of a neutral element, and moreover some $n$-ary groups do not contain an
identity at all.
\end{enumerate}

\subsection{\label{ck-subsec-dp}Polyadic direct power construction}

Let us consider the $m$-ary semigroup $\mathcal{S}^{\left(  m\right)
}=\left\langle S\mid\mu^{\left(  m\right)  }\mid total\ assoc\right\rangle $,
where $S$ is the underlying set, and $\mu^{\left(  m\right)  }:S^{\times
m}\rightarrow S$ is the (totally associative) $m$-ary multiplication in
$\mathcal{S}$. The Cartesian product of two underlying sets $S^{\prime
}=S^{\times m}$ can be endowed with an associative $n$-ary multiplication in
various different ways, and also it can be that $n\neq m$ \cite{dup2022}.
Again, we introduce the doubles $\mathbf{S}=\left(
\begin{array}
[c]{c}%
a\\
b
\end{array}
\right)  \in S\times S$, and the polyadic direct product which is the $n$-ary
semigroup of doubles $\mathcal{S}^{\prime\left(  n\right)  }=\left\langle
S^{\prime}\mid\mathbf{\mu}^{\prime\left(  n\right)  }\mid
total\ assoc\right\rangle $.

There are two possibilities to build the polyadic associative direct product
\cite{dup2022}:

\begin{enumerate}
\item The componentwise construction which corresponds to the standard binary
direct product (\ref{ck-mg12}).

\item The noncomponentwise one, which corresponds to the twisted direct
product (\ref{ck-mg12t}).
\end{enumerate}

In the first case, the $n$-ary semigroup of doubles $\mathcal{S}%
^{\prime\left(  n\right)  }$ corresponds to the full polyadic external product
of \cite{dup2022}.

\begin{definition}
An \textit{ }$n$-ary \textit{componentwise} \textit{direct product} (power)
semigroup of doubles consists of two $n$-ary semigroups $\mathcal{S}%
^{\prime\left(  n\right)  }=\mathcal{S}^{\left(  n\right)  }\times
\mathcal{S}^{\left(  n\right)  }$ (of \textsf{the same} arity)%
\begin{equation}
\mathbf{\mu}^{\prime\left(  n\right)  }\left[  \mathbf{S}_{1},\mathbf{S}%
_{2},\ldots,\mathbf{S}_{n}\right]  =\left(
\begin{array}
[c]{c}%
\mu^{\left(  n\right)  }\left[  a_{1},a_{2},\ldots,a_{n}\right] \\[5pt]%
\mu^{\left(  n\right)  }\left[  b_{1},b_{2},\ldots,b_{n}\right]
\end{array}
\right)  ,\ \ \ \ \mathbf{S}_{i}=\left(
\begin{array}
[c]{c}%
a_{i}\\
b_{i}%
\end{array}
\right)  ,\ \ a_{i},b_{i}\in S,\ \label{ck-mg1}%
\end{equation}
where the (total) polyadic associativity (\ref{ck-ghu}) of $\mathbf{\mu
}^{\prime\left(  n\right)  }$ is governed by the associativity of the
constituent semigroup $\mathcal{S}^{\left(  n\right)  }$.
\end{definition}

The simplest case is give by

\begin{definition}
\label{ck-def-der}A \textit{ full polyadic direct product (power)}
$\mathcal{S}^{\prime\left(  n\right)  }=\mathcal{S}^{\left(  n\right)  }%
\times\mathcal{S}^{\left(  n\right)  }$ is called \textit{derived}, if the
constituent $n$-ary semigroup $\mathcal{S}^{\left(  n\right)  }$ is derived,
such that its operation $\mu^{\left(  n\right)  }$ is composition of the
binary operations $\mu^{\left(  2\right)  }$.
\end{definition}

In the derived case all the operations in (\ref{ck-mg1}) have the form (see
(\ref{ck-mn})--(\ref{ck-n}))%
\begin{equation}
\mu^{\left(  n\right)  }=\left(  \mu^{\left(  2\right)  }\right)
^{\circ\left(  n-1\right)  },\mu^{\left(  2\right)  }=\left(  \cdot\right)
,\ \ \ \mathbf{\mu}^{\prime\left(  n\right)  }=\left(  \mathbf{\mu}%
^{\prime\left(  2\right)  }\right)  ^{\circ\left(  n-1\right)  },\mathbf{\mu
}^{\left(  2\right)  }=\left(  \bullet^{\prime}\right)  . \label{ck-m2}%
\end{equation}

The operations of the derived polyadic direct product can be written as (cf.
the binary case (\ref{ck-mg12}))%
\begin{equation}
\mathbf{\mu}_{der}^{\prime\left(  n\right)  }\left[  \mathbf{S}_{1}%
,\mathbf{S}_{2},\ldots,\mathbf{S}_{n}\right]  =\mathbf{S}_{1}\bullet^{\prime
}\mathbf{S}_{2}\bullet^{\prime}\ldots\bullet^{\prime}\mathbf{S}_{n}=\left(
\begin{array}
[c]{c}%
a_{1}\cdot a_{2}\cdot\ldots\cdot a_{n}\\[5pt]%
b_{1}\cdot b_{2}\cdot\ldots\cdot b_{n}%
\end{array}
\right)  ,
\end{equation}
and so it is simply a repitition of the binary products (\ref{ck-mg12}).
Therefore, it would be more interesting to consider nonderived polyadic
analogs of the direct product which do not come down to the binary ones.

The nonderived version of the polyadic direct product, the hetero
(\textquotedblleft entangled\textquotedblright) product, was introduced in
\cite{dup2022} for an arbitrary number of constituents. Here we apply it for
two $m$-ary semigroups to get the nonderived associative direct product, which
can have a different arity $n\neq m$.

The general structure of the hetero product formally coincides
(\textquotedblleft reversely\textquotedblright)\ with the main heteromorphism
equation \cite{dup2018a}. The additional parameter which determines the arity
$n$ of the hetero power of the initial $m$-ary semigroup is the number of
intact elements $\ell_{\operatorname*{id}}$ and number of constituents $k$. In
our case of $k=2$ multipliers, we have only two possibilities $\ell
_{\operatorname*{id}}=0,1$. Thus, we arrive at

\begin{definition}
The \textit{hetero }(\textit{\textquotedblleft entangled\textquotedblright%
})\textit{ power} $\boxtimes$\textit{ (square)} of the $m$-ary semigroup
$\mathcal{S}^{\left(  m\right)  }=\left\langle S\mid\mu^{\left(  m\right)
}\right\rangle $ is the $n$-ary semigroup defined on the Cartesian power
$S^{\prime}=S\times S$, such that $\mathcal{S}^{\prime\left(  n\right)
}=\left\langle S^{\prime}\mid\mathbf{\mu}^{\prime\left(  n\right)
}\right\rangle $,%
\begin{equation}
\mathcal{S}^{\prime\left(  n\right)  }=\mathcal{S}^{\left(  m\right)
}\boxtimes\mathcal{S}^{\left(  m\right)  },
\end{equation}
and the $n$-ary multiplication of doubles $\mathbf{S}_{ij}=\left(
\begin{array}
[c]{c}%
a_{i}\\
a_{j}%
\end{array}
\right)  \in S\times S$, $a_{i},a_{j}\in S$, is given (informally) by%
\begin{equation}
\mathbf{\mu}^{\prime\left(  n\right)  }\left[  \left(
\begin{array}
[c]{c}%
a_{1}\\
a_{2}%
\end{array}
\right)  ,\ldots,\left(
\begin{array}
[c]{c}%
a_{2n-1}\\
a_{2n}%
\end{array}
\right)  \right]  =\left\{
\genfrac{}{}{0pt}{0}{\left(
\begin{array}
[c]{c}%
\mu^{\left(  m\right)  }\left[  a_{1},\ldots,a_{m}\right]  ,\\
\mu^{\left(  m\right)  }\left[  a_{m+1},\ldots,a_{2m}\right]
\end{array}
\right)  ,\ \ \ \ell_{\operatorname*{id}}=0,\ \ n=m,}{\left(
\begin{array}
[c]{c}%
\mu^{\left(  m\right)  }\left[  a_{1},\ldots,a_{m}\right]  ,\\
a_{m+1}%
\end{array}
\right)  ,\ \ \ \ell_{\operatorname*{id}}=1,\ \ n=\frac{m+1}{2},}%
\right.  , \label{ck-mmnn}%
\end{equation}
where $\ell_{\operatorname*{id}}=0,1$ is the number of \textit{intact
elements} in the r.h.s. of the polyadic direct product. The hetero power
arities are connected by the \textit{arity changing formula} (with $k=2$)
\cite{dup2018a}%
\begin{equation}
n=m-\dfrac{m-1}{2}\ell_{\operatorname*{id}}, \label{ck-nnk}%
\end{equation}
with the integer $\dfrac{n-1}{2}\ell_{\operatorname*{id}}\geq1$.
\end{definition}

In the case $\ell_{\operatorname*{id}}=1$, the initial and final arities $m$
and $n$ are not arbitrary, but \textquotedblleft quantized\textquotedblright%
\ such that the fraction in (\ref{ck-nnk}) has to be an integer (see the first
row in \textsc{Table 1} of \cite{dup2022})%
\begin{equation}%
\begin{tabular}
[c]{ccccc}%
$m=$ & $3,$ & $5,$ & $7,$ & $\ldots$\\
$n=$ & $2,$ & $3,$ & $4,$ & $\ldots$%
\end{tabular}
\ \label{ck-m32}%
\end{equation}

The concrete placement of elements and multiplications in (\ref{ck-mmnn}) to
obtain the associative $\mathbf{\mu}^{\prime\left(  n^{\prime}\right)  }$ is
governed by the associativity quiver technique \cite{dup2018a}.

Thus, the classification of the hetero powers for a nonbinary initial
semigroup $m\geq3$ consists of two limiting cases:

\begin{enumerate}
\item \textit{Intactless power}: there are no intact elements $\ell
_{\operatorname*{id}}=0$. The arity of the hetero power reaches its
\textsf{maximum} and coincides with the arity of the initial semigroup $n=m$.
Since $m\geq3$, there is no binary hetero product.

\item \textit{Binary power}: the final semigroup is of lowest arity, i.e.
binary $n=2$, and it follows from (\ref{ck-m32}), that the only possibility is
$m=3$.
\end{enumerate}

Let us consider some concrete examples of the hetero powers.

\begin{example}
\label{ck-ex-m2}Let $\mathcal{S}^{\left(  3\right)  }=\left\langle S\mid
\mu^{\left(  3\right)  }\right\rangle $ be a commutative ternary semigroup,
then we can construct its square of the doubles $\mathbf{S}=\left(
\begin{array}
[c]{c}%
a\\
b
\end{array}
\right)  \in S\times S$ in two ways to obtain the associative hetero power%
\begin{equation}
\mathbf{\mu}^{\prime\left(  2\right)  }\left[  \mathbf{S}_{1},\mathbf{S}%
_{2}\right]  =\mathbf{\mu}^{\prime\left(  2\right)  }\left[  \left(
\begin{array}
[c]{c}%
a_{1}\\
b_{1}%
\end{array}
\right)  ,\left(
\begin{array}
[c]{c}%
a_{2}\\
b_{2}%
\end{array}
\right)  \right]  \left\{
\begin{array}
[c]{c}%
\left(
\begin{array}
[c]{c}%
\mu^{\left(  3\right)  }\left[  a_{1},b_{1},a_{2}\right] \\
b_{2}%
\end{array}
\right)  ,\\
\left(
\begin{array}
[c]{c}%
\mu^{\left(  3\right)  }\left[  a_{1},b_{2},a_{2}\right] \\
b_{1}%
\end{array}
\right)  ,
\end{array}
\right.  \ \ \ a_{i},b_{i}\in S. \label{ck-m2g}%
\end{equation}
This means that the Cartesian square $S\times S$ can be endowed with an
associative multiplication $\mathbf{\mu}^{\prime\left(  2\right)  }$ in two
ways, and therefore $\mathcal{S}^{\prime\left(  2\right)  }=\left\langle
S^{\prime}\mid\mathbf{\mu}^{\prime\left(  2\right)  }\right\rangle $ is a
noncommutative semigroup, being the hetero power $\mathcal{S}^{\prime\left(
2\right)  }=\mathcal{S}^{\left(  3\right)  }\boxtimes\mathcal{S}^{\left(
3\right)  }$. If the ternary semigroup $\mathcal{S}^{\left(  3\right)  }$ has
a ternary identity $e\in S$, then $\mathcal{S}^{\prime\left(  2\right)  }$ has
only the left (right) identity $\mathbf{E}=\left(
\begin{array}
[c]{c}%
e\\
e
\end{array}
\right)  \in S\times S$, since $\mathbf{\mu}^{\prime\left(  2\right)  }\left[
\mathbf{E},\mathbf{S}\right]  =\mathbf{S}$ ($\mathbf{\mu}^{\prime\left(
2\right)  }\left[  \mathbf{S},\mathbf{E}\right]  =\mathbf{S}$), but not the
right (left) identity. Thus, $\mathcal{S}^{\prime\left(  2\right)  }$ can be a
semigroup only, even in the case when $\mathcal{S}^{\left(  3\right)  }$ is a
ternary group.
\end{example}

\begin{example}
\label{ck-ex-m3}Take $\mathcal{S}^{\left(  3\right)  }=\left\langle S\mid
\mu^{\left(  3\right)  }\right\rangle $ to be a commutative ternary semigroup,
then the multiplication of the doubles is ternary and is
\textsf{noncomponentwise} (as opposed to the \textsf{componentwise} product
(\ref{ck-mg1}))%
\begin{equation}
\mathbf{\mu}^{\prime\left(  3\right)  }\left[  \mathbf{S}_{1},\mathbf{S}%
_{2},\mathbf{S}_{3}\right]  =\mathbf{\mu}^{\prime\left(  3\right)  }\left[
\left(
\begin{array}
[c]{c}%
a_{1}\\
b_{1}%
\end{array}
\right)  ,\left(
\begin{array}
[c]{c}%
a_{2}\\
b_{2}%
\end{array}
\right)  ,\left(
\begin{array}
[c]{c}%
a_{3}\\
b_{3}%
\end{array}
\right)  \right]  =\left(
\begin{array}
[c]{c}%
\mu^{\left(  3\right)  }\left[  a_{1},b_{2},a_{3}\right] \\
\mu^{\left(  3\right)  }\left[  b_{1},a_{2},b_{3}\right]
\end{array}
\right)  ,\ \ \ a_{i},b_{i}\in S. \label{ck-ncom}%
\end{equation}
Nevertheless $\mathbf{\mu}^{\prime\left(  3\right)  }$ is associative (and is
described by the Post-like associative quiver \cite{dup2018a}), therefore the
hetero power $\mathcal{S}^{\prime\left(  3\right)  }=\mathcal{S}^{\left(
3\right)  }\boxtimes\mathcal{S}^{\left(  3\right)  }$ is indeed the
noncommutative ternary semigroup $\mathcal{S}^{\prime\left(  3\right)
}=\left\langle G\times G\mid\mathbf{\mu}^{\prime\left(  3\right)
}\right\rangle $. In this case, as opposed to the previous example, the
existence of the ternary identity $e$ in $\mathcal{S}^{\left(  3\right)  }$
implies the ternary identity in the direct product $\mathcal{S}^{\prime\left(
3\right)  }$ by $\mathbf{E}=\left(
\begin{array}
[c]{c}%
e\\
e
\end{array}
\right)  $, such that%
\begin{equation}
\mathbf{\mu}^{\prime\left(  3\right)  }\left[  \mathbf{S},\mathbf{E}%
,\mathbf{E}\right]  =\mathbf{\mu}^{\prime\left(  3\right)  }\left[
\mathbf{E},\mathbf{S},\mathbf{E}\right]  =\mathbf{\mu}^{\prime\left(
3\right)  }\left[  \mathbf{E},\mathbf{E,S}\right]  =\mathbf{S}.
\end{equation}

\end{example}

\begin{proposition}
If the initial $m$-ary semigroup $\mathcal{S}^{\left(  m\right)  }$ contains
an identity, then the hetero square $n$-ary semigroup $\mathcal{S}%
^{\prime\left(  n\right)  }=\mathcal{S}^{\left(  m\right)  }\boxtimes
\mathcal{S}^{\left(  m\right)  }$ can contain an identity only in the
\textsf{intactless case} and the Post-like quiver \cite{dup2018a}. For the
binary power $n=2$ only the one-sided identity is possible.
\end{proposition}

Next we consider more complicated hetero power (\textquotedblleft
entangled\textquotedblright) constructions with and without intact elements,
which is possible for $m\geq4$ only \cite{dup2018a}.

\begin{example}
Let $\mathcal{S}^{\left(  5\right)  }=\left\langle S\mid\mu^{\left(  5\right)
}\right\rangle $ be a $5$-ary semigroup, then we construct its $5$-ary totally
associative hetero square $\mathcal{G}^{\prime\left(  5\right)  }=\left\langle
G^{\prime}\mid\mathbf{\mu}^{\prime\left(  5\right)  }\right\rangle $ using the
Post-like associative quiver without intact elements. We define the $5$-ary
multiplication of the doubles by%
\begin{equation}
\mathbf{\mu}^{\prime\left(  5\right)  }\left[  \mathbf{S}_{1},\mathbf{S}%
_{2},\mathbf{S}_{3},\mathbf{S}_{4},\mathbf{S}_{5}\right]  =\left(
\begin{array}
[c]{c}%
\mu^{\left(  5\right)  }\left[  a_{1},b_{2},a_{3},b_{4},a_{5}\right] \\[5pt]%
\mu^{\left(  5\right)  }\left[  b_{1},a_{2},b_{3},a_{4},b_{5}\right]
\end{array}
\right)  ,\ \ \ a_{i},b_{i}\in S. \label{ck-m4}%
\end{equation}
It can be shown that $\mathbf{\mu}^{\prime\left(  5\right)  }$ is totally
associative, and therefore $\mathcal{S}^{\prime\left(  5\right)
}=\left\langle S^{\prime}\mid\mathbf{\mu}^{\prime\left(  5\right)
}\right\rangle $ is a $5$-ary \textsf{commutative} semigroup. If
$\mathcal{S}^{\left(  5\right)  }$ has the $5$-ary identity $e$ satisfying%
\begin{equation}
\mu^{\left(  5\right)  }\left[  e,e,e,e,a\right]  =a,\ \ \ \forall a\in
S,\ \ e\in S, \label{ck-m4e}%
\end{equation}
then the hetero power $\mathcal{S}^{\prime\left(  5\right)  }$ has the $5$-ary
identity%
\begin{equation}
\mathbf{E}=\left(
\begin{array}
[c]{c}%
e\\
e
\end{array}
\right)  ,\ \ \ e\in S. \label{ck-ee}%
\end{equation}

\end{example}

A more nontrivial example is a hetero product (power) which has a different
arity than that of the initial semigroup.

\begin{example}
Let $\mathcal{S}^{\left(  5\right)  }=\left\langle S\mid\mu^{\left(  5\right)
}\right\rangle $ be a commutative $5$-ary semigroup, then we can construct its
ternary associative hetero power $\mathcal{S}^{\prime\left(  3\right)
}=\left\langle S^{\prime}\mid\mathbf{\mu}^{\prime\left(  3\right)
}\right\rangle $ using the associative quivers with one intact element (see
(\ref{ck-m32}) for allowed \textquotedblleft quantized\textquotedblright%
\ arities). We propose for the triples $\mathbf{S}_{i}$ the following ternary
multiplication%
\begin{equation}
\mathbf{\mu}^{\prime\left(  3\right)  }\left[  \mathbf{S}_{1},\mathbf{S}%
_{2},\mathbf{S}_{3}\right]  =\mathbf{\mu}^{\prime\left(  3\right)  }\left[
\left(
\begin{array}
[c]{c}%
a_{1}\\
b_{1}%
\end{array}
\right)  ,\left(
\begin{array}
[c]{c}%
a_{2}\\
b_{2}%
\end{array}
\right)  ,\left(
\begin{array}
[c]{c}%
a_{3}\\
b_{3}%
\end{array}
\right)  \right]  =\left(
\begin{array}
[c]{c}%
\mu^{\left(  5\right)  }\left[  a_{1},b_{2},a_{3},b_{1},a_{2}\right] \\[5pt]%
b_{3}%
\end{array}
\right)  ,\ \ \ a_{i},b_{i}\in S. \label{ck-m3}%
\end{equation}
It can be seen that $\mathbf{\mu}^{\prime\left(  3\right)  }$ is totally
associative, and therefore the hetero power of $5$-ary semigroup
$\mathcal{S}^{\left(  5\right)  }=\left\langle S\mid\mu^{\left(  5\right)
}\right\rangle $ is a \textsf{noncommutative} ternary semigroup $\mathcal{S}%
^{\prime\left(  3\right)  }=\left\langle S^{\prime}\mid\mathbf{\mu}%
^{\prime\left(  3\right)  }\right\rangle $, such that $\mathcal{S}%
^{\prime\left(  3\right)  }=\mathcal{S}^{\left(  5\right)  }\boxtimes
\mathcal{S}^{\left(  5\right)  }$. If the initial $5$-ary semigroup
$\mathcal{S}^{\left(  5\right)  }$ has the identity satisfying%
\begin{equation}
\mu^{\left(  5\right)  }\left[  e,e,e,e,a\right]  =a,\ \ \ \forall a\in
S,\ \ e\in S,
\end{equation}
then the ternary hetero power $\mathcal{S}^{\prime\left(  3\right)  }$ has
\textsf{only} the left ternary identity (\ref{ck-ee}) satisfying one relation%
\begin{equation}
\mathbf{\mu}^{\prime\left(  3\right)  }\left[  \mathbf{E},\mathbf{E}%
,\mathbf{S}\right]  =\mathbf{S},\ \ \ \forall\mathbf{S}\in S\times S,
\end{equation}
and therefore $\mathcal{S}^{\prime\left(  3\right)  }$ is a ternary semigroup
with the left identity.
\end{example}

\subsection{Equivalence relations for $m$-ary semigroups}

Consider of the extension of the equivalence relations (\ref{ck-m2a}) and
(\ref{ck-m2b}) for $m$-ary semigroups. On the polyadic direct product (square)
$\mathcal{S}^{\left(  m\right)  }\times\mathcal{S}^{\left(  m\right)  }$ we
define two kinds of corresponding binary relations.

\begin{definition}
The first relation $\left(  \sim_{gauge}\right)  $ is described by the
\textit{polyadic \textquotedblleft gauge\textquotedblright\ shifts} of the
doubles $\mathbf{S=}\left(
\begin{array}
[c]{c}%
a\\
b
\end{array}
\right)  \in S\times S$ in the form%
\begin{equation}
\left(
\begin{array}
[c]{c}%
a_{1}\\
b_{1}%
\end{array}
\right)  \sim_{gauge}\left(
\begin{array}
[c]{c}%
a_{2}\\
b_{2}%
\end{array}
\right)  \Longleftrightarrow\exists_{x,y\in S}\ \left(
\begin{array}
[c]{c}%
\mu^{\left(  m\right)  }\left[  a_{1}^{m-1},x\right] \\[5pt]%
\mu^{\left(  m\right)  }\left[  b_{1}^{m-1},x\right]
\end{array}
\right)  =\left(
\begin{array}
[c]{c}%
\mu^{\left(  m\right)  }\left[  a_{2}^{m-1},y\right] \\[5pt]%
\mu^{\left(  m\right)  }\left[  b_{2}^{m-1},y\right]
\end{array}
\right)  ,\ \ \forall a_{i},b_{i}\in S. \label{ck-eq1}%
\end{equation}

The second relation $\left(  \sim_{2}\right)  $\ is given by the
\textit{polyadic \textquotedblleft twisted\textquotedblright} shift%
\begin{equation}
\left(
\begin{array}
[c]{c}%
a_{1}\\
b_{1}%
\end{array}
\right)  \sim_{twist}\left(
\begin{array}
[c]{c}%
a_{2}\\
b_{2}%
\end{array}
\right)  \Longleftrightarrow\exists_{z\in S}\ \left(  \mu^{\left(  m\right)
}\right)  ^{\circ2}\left[  a_{1}^{m-1},b_{2}^{m-1},z\right]  =\left(
\mu^{\left(  m\right)  }\right)  ^{\circ2}\left[  a_{2}^{m-1},b_{1}%
^{m-1},z\right]  ,\ \ \forall a_{i},b_{i}\in S. \label{ck-eq2}%
\end{equation}

\end{definition}

\begin{proposition}
The relations described by the polyadic \textquotedblleft
gauge\textquotedblright\ shifts and \textquotedblleft
twisted\textquotedblright\ shift coincide $\sim_{gauge}=\sim_{twist}\equiv
\sim_{m}$.
\end{proposition}

\begin{proof}
$\sim_{twist}\Rightarrow\sim_{gauge}$) Let (\ref{ck-eq2}) hold, then putting
$x=\mu^{\left(  m\right)  }\left[  b_{2}^{m-1},z\right]  $, $y=\mu^{\left(
m\right)  }\left[  a_{1}^{m-1},z\right]  $, we get (\ref{ck-eq1}).

$\sim_{gauge}\Rightarrow\sim_{twist}$) Conversely, if (\ref{ck-eq1}) takes
place, then we can always find $t_{1},\ldots,t_{m-2}\in S$, such that
$z=\mu^{\left(  m\right)  }\left[  x,y,t_{1},\ldots t_{m-2}\right]  $ and%
\begin{align}
&  \left(  \mu^{\left(  m\right)  }\right)  ^{\circ2}\left[  a_{1}^{m-1}%
,b_{2}^{m-1},z\right]  =\mu^{\left(  m\right)  }\left[  \mu^{\left(  m\right)
}\left[  a_{1}^{m-1},x\right]  ,\mu^{\left(  m\right)  }\left[  b_{2}%
^{m-1},y\right]  ,t_{1},\ldots t_{m-2}\right] \nonumber\\
&  =\mu^{\left(  m\right)  }\left[  \mu^{\left(  m\right)  }\left[
a_{2}^{m-1},x\right]  ,\mu^{\left(  m\right)  }\left[  b_{1}^{m-1},y\right]
,t_{1},\ldots t_{m-2}\right]  =\left(  \mu^{\left(  m\right)  }\right)
^{\circ2}\left[  a_{2}^{m-1},b_{1}^{m-1},z\right]  .
\end{align}

\end{proof}

\begin{proposition}
\label{ck-prop-bineq}The relation $\sim_{m}$ is the equivalence relation on
the set of doubles $\left\{  \mathbf{S}\right\}  $, $\mathbf{S}\in S\times S$.
\end{proposition}

\begin{proof}
\textbf{1}) \textsf{Reflexivity} ($1=2$) of (\ref{ck-eq1}) is obvious, and for
$\sim_{twist}$ it follows from (\ref{ck-eq2}), becoming the identity $\left(
\mu^{\left(  m\right)  }\right)  ^{\circ2}\left[  a^{m-1},b^{m-1},z\right]
=\left(  \mu^{\left(  m\right)  }\right)  ^{\circ2}\left[  a^{m-1}%
,b^{m-1},z\right]  $.

\textbf{2}) \textsf{Symmetry} ($1\longleftrightarrow2$) is evident for
$\sim_{gauge}$, while (\ref{ck-eq2}) is symmetric with respect to
($1\longleftrightarrow2$).

\textbf{3}) \textsf{Transitivity}. Let $\mathbf{S}_{1}\sim_{m}\mathbf{S}_{2}$
and $\mathbf{S}_{2}\sim_{m}\mathbf{S}_{3}$, and we will prove that
$\mathbf{S}_{1}\sim_{m}\mathbf{S}_{3}$, $\mathbf{S}_{i}\in S\times S$.

We start from the \textquotedblleft gauge\textquotedblright-like relations
$\mathbf{S}_{1}\sim_{gauge}\mathbf{S}_{2}$ and $\mathbf{S}_{2}\sim
_{gauge}\mathbf{S}_{3}$ for the first components of the doubles%
\begin{align}
\mu^{\left(  m\right)  }\left[  a_{1}^{m-1},x_{1}\right]   &  =\mu^{\left(
m\right)  }\left[  a_{2}^{m-1},y_{1}\right]  ,\label{ck-mmm1}\\
\mu^{\left(  m\right)  }\left[  a_{2}^{m-1},x_{2}\right]   &  =\mu^{\left(
m\right)  }\left[  a_{3}^{m-1},y_{2}\right]  ,\ \ \ \exists x_{i},y_{i}\in S.
\label{ck-mmm2}%
\end{align}

Then we multiply separately the left hand sides and right hand sides of
(\ref{ck-mmm1}), (\ref{ck-mmm2}) together with $\left(  m-2\right)  $
identities $t_{1}=t_{1}$, $t_{2}=t_{2},\ldots,t_{m-2}=t_{m-2}$, $t_{1}%
,\ldots,t_{m-2}\in S$, and derive%
\begin{equation}
\left(  \mu^{\left(  m\right)  }\right)  ^{\circ3}\left[  a_{1}^{m-1}%
,x_{1},a_{2}^{m-1},x_{2},t_{1},\ldots,t_{m-2}\right]  =\left(  \mu^{\left(
m\right)  }\right)  ^{\circ3}\left[  a_{2}^{m-1},y_{1},a_{3}^{m-1},y_{2}%
,t_{1},\ldots,t_{m-2}\right]  .
\end{equation}

Denoting $x_{3}=\left(  \mu^{\left(  m\right)  }\right)  ^{\circ2}\left[
a_{2}^{m-1},x_{1},x_{2},t_{1},\ldots,t_{m-2}\right]  $ and $y_{3}=\left(
\mu^{\left(  m\right)  }\right)  ^{\circ2}\left[  a_{2}^{m-1},y_{1}%
,y_{2},t_{1},\ldots,t_{m-2}\right]  $, we obtain in the form%
\begin{equation}
\mu^{\left(  m\right)  }\left[  a_{1}^{m-1},x_{3}\right]  =\mu^{\left(
m\right)  }\left[  a_{3}^{m-1},y_{3}\right]  .
\end{equation}
The second components of the doubles can be treated similarly%
\begin{equation}
\mu^{\left(  m\right)  }\left[  b_{1}^{m-1},x_{3}\right]  =\mu^{\left(
m\right)  }\left[  b_{3}^{m-1},y_{3}\right]  ,
\end{equation}
and so we have $\mathbf{S}_{1}\sim_{gauge}\mathbf{S}_{3}$.

The \textquotedblleft twisted\textquotedblright-like relations $\mathbf{S}%
_{1}\sim_{twist}\mathbf{S}_{2}$ and $\mathbf{S}_{2}\sim_{twist}\mathbf{S}_{3}$
have the form%
\begin{align}
\left(  \mu^{\left(  m\right)  }\right)  ^{\circ2}\left[  a_{1}^{m-1}%
,b_{2}^{m-1},z_{1}\right]   &  =\left(  \mu^{\left(  m\right)  }\right)
^{\circ2}\left[  a_{2}^{m-1},b_{1}^{m-1},z_{1}\right]  ,\label{ck-mm1}\\
\left(  \mu^{\left(  m\right)  }\right)  ^{\circ2}\left[  a_{2}^{m-1}%
,b_{3}^{m-1},z_{2}\right]   &  =\left(  \mu^{\left(  m\right)  }\right)
^{\circ2}\left[  a_{3}^{m-1},b_{2}^{m-1},z_{2}\right]  . \label{ck-mm2}%
\end{align}
We multiply separately the left hand sides and right hand sides of
(\ref{ck-mm1}), (\ref{ck-mm2}) together with $\left(  m-2\right)  $ identities
$t_{1}=t_{1}$, $t_{2}=t_{2},\ldots,t_{m-2}=t_{m-2}$, $t_{1},\ldots,t_{m-2}\in
S$, to get%
\begin{align}
&  \left(  \mu^{\left(  m\right)  }\right)  ^{\circ5}\left[  a_{1}^{m-1}%
,b_{2}^{m-1},z_{1},a_{2}^{m-1},b_{3}^{m-1},z_{2},t_{1},\ldots,t_{m-2}\right]
\nonumber\\
&  =\left(  \mu^{\left(  m\right)  }\right)  ^{\circ5}\left[  a_{2}%
^{m-1},b_{1}^{m-1},z_{1},a_{3}^{m-1},b_{2}^{m-1},z_{2},t_{1},\ldots
,t_{m-2}\right]  ,
\end{align}
which can be written as%
\begin{equation}
\left(  \mu^{\left(  m\right)  }\right)  ^{\circ2}\left[  a_{1}^{m-1}%
,b_{3}^{m-1},z_{3}\right]  =\left(  \mu^{\left(  m\right)  }\right)  ^{\circ
2}\left[  a_{3}^{m-1},b_{1}^{m-1},z_{3}\right]  , \label{ck-mm3}%
\end{equation}
where $z_{3}=\left(  \mu^{\left(  m\right)  }\right)  ^{\circ3}\left[
a_{2}^{m-1},b_{2}^{m-1},z_{1},z_{2},t_{1},\ldots,t_{m-2}\right]  $, $z_{3}\in
S.$ Therefore, it follows from (\ref{ck-mm3}), that $\mathbf{S}_{1}%
\sim_{twist}\mathbf{S}_{3}$.
\end{proof}

We next introduce the corresponding equivalence classes.

\begin{definition}
The \textit{equivalence class} of doubles corresponding to the equivalence
relation $\sim_{m}$ on the $m$-ary semigroup $\mathcal{S}^{\left(  m\right)
}$ is the factor%
\begin{equation}
\widetilde{\mathfrak{S}}=\left[
\begin{array}
[c]{c}%
a\\
b
\end{array}
\right]  _{m}=\left\{  \mathbf{S}\right\}  /\sim_{m},\ \ \ \ \ \mathbf{S}%
=\left(
\begin{array}
[c]{c}%
a\\
b
\end{array}
\right)  \in S\times S, \label{ck-sbin}%
\end{equation}
where $\sim_{m}$ is defined in (\ref{ck-eq1}) or (\ref{ck-eq2}).
\end{definition}

Next we consider the structure of the equivalence classes (\ref{ck-sbin}) in
some examples.

\begin{example}
[\textsf{Negative numbers}]\label{ck-ex-neg}Let $S=\mathbb{N}_{-}%
\subset\mathbb{Z}$ be the set of negative (integer) numbers (without zero),
which do not form a binary semigroup, obviously. However, we can introduce the
ternary multiplication $\mu^{\left(  3\right)  }\left[  a,b,c\right]  =abc$
(ordinary product in $\mathbb{Z}$), such that $\mathcal{S}_{neg}^{\left(
3\right)  }=\left\langle \mathbb{N}_{-}\mid\mu^{\left(  3\right)
}\right\rangle $ becomes the ternary semigroup (of negative numbers). The
doubles have the form%
\begin{equation}
\mathbf{S}_{neg}=\left(
\begin{array}
[c]{c}%
-p\\
-q
\end{array}
\right)  \in\mathbb{N}_{-}\times\mathbb{N}_{-},\ \ \ p,q\in\mathbb{N}.
\label{ck-sneg}%
\end{equation}
The equivalence relations (\ref{ck-eq1}) and (\ref{ck-eq2}) become ($m=3$)%
\begin{equation}
\left(
\begin{array}
[c]{c}%
-p_{1}\\
-q_{1}%
\end{array}
\right)  \sim_{gauge}\left(
\begin{array}
[c]{c}%
-p_{2}\\
-q_{2}%
\end{array}
\right)  \Longleftrightarrow\exists_{x,y\in\mathbb{N}}\ \left(
\begin{array}
[c]{c}%
-p_{1}^{2}x\\
-q_{1}^{2}x
\end{array}
\right)  =\left(
\begin{array}
[c]{c}%
-p_{2}^{2}y\\
-q_{2}^{2}y
\end{array}
\right)  ,\ \ \ p_{i},q_{i}\in\mathbb{N},
\end{equation}
and%
\begin{equation}
\left(
\begin{array}
[c]{c}%
-p_{1}\\
-q_{1}%
\end{array}
\right)  \sim_{twist}\left(
\begin{array}
[c]{c}%
-p_{2}\\
-q_{2}%
\end{array}
\right)  \Longleftrightarrow\exists_{x,y\in\mathbb{N}}\ p_{1}^{2}q_{2}%
^{2}z=p_{2}^{2}q_{1}^{2}z. \label{ck-pqz}%
\end{equation}

For instance, we compute%
\begin{align}
\left(
\begin{array}
[c]{c}%
-1\\
-2
\end{array}
\right)   &  \sim\left(
\begin{array}
[c]{c}%
-2\\
-4
\end{array}
\right)  \sim\left(
\begin{array}
[c]{c}%
-3\\
-6
\end{array}
\right)  \ldots,\\
\left(
\begin{array}
[c]{c}%
-4\\
-3
\end{array}
\right)   &  \sim\left(
\begin{array}
[c]{c}%
-8\\
-6
\end{array}
\right)  \sim\left(
\begin{array}
[c]{c}%
-12\\
-9
\end{array}
\right)  \ldots.
\end{align}

In general, the two elements
\begin{equation}
\left(
\begin{array}
[c]{c}%
-p\\
-q
\end{array}
\right)  ,\left(
\begin{array}
[c]{c}%
-kp\\
-kq
\end{array}
\right)  \label{ck-kp}%
\end{equation}
are in the same class, $k\in\mathbb{N}$. Each equivalence class $\widetilde
{\mathfrak{S}}_{neg}$ contains the maximal representative $\left[
\begin{array}
[c]{c}%
-p_{\min}\\
-q_{\min}%
\end{array}
\right]  _{3}$ corresponding to mininal $p=p_{\min}$ and $q=q_{\min}$,
together with other elements (\ref{ck-kp}) of the class. Thus, for the ternary
semigroup of negative integers the structure of the binary equivalence classes
in $\mathcal{S}_{neg}^{\left(  3\right)  }\times\mathcal{S}_{neg}^{\left(
3\right)  }\diagup\sim_{3}$ is given by the following maximal representatives%
\begin{align}
\widetilde{\mathfrak{S}}_{neg,p>q}  &  =\left\{  \left[
\begin{array}
[c]{c}%
-p_{\min}\\
-q_{\min}%
\end{array}
\right]  _{3}\right\}  ,\ \ \ \ p_{\min}>q_{\min},\\
\widetilde{\mathfrak{S}}_{neg,p<q}  &  =\left\{  \left[
\begin{array}
[c]{c}%
-p_{\min}\\
-q_{\min}%
\end{array}
\right]  _{3}\right\}  ,\ \ \ \ p_{\min}<q_{\min},\ \ \ \ p_{\min},q_{\min}%
\in\mathbb{N},
\end{align}
where $p_{\min}\neq q_{\min}$, and they are mutually prime $\gcd\left(
p_{\min},q_{\min}\right)  =1$.
\end{example}

\subsection{Polyadic group completion}

In general, the procedure of group completion consists of two steps:

\begin{enumerate}
\item Endow the set of equivalence classes $\left\{  \widetilde{\mathfrak{S}%
}\right\}  $ with the multiplication $\widetilde{\mathbf{\mu}}_{S}^{\left(
\tilde{n}\right)  }$ (in the binary case (\ref{ck-mw})).

\item Find (a polyadic analog of) an inverse and construct a group (as in
(\ref{ck-g1})).
\end{enumerate}

It is commonly accepted that the multiplication of classes inherits the
product of representatives. So the informal formula of class multiplication
coincides with the hetero (\textquotedblleft entangled\textquotedblright)
power (\ref{ck-mmnn}), indeed%
\begin{equation}
\widetilde{\mathbf{\mu}}_{S}^{\left(  \tilde{n}\right)  }\left[  \left[
\begin{array}
[c]{c}%
a_{1}\\
a_{2}%
\end{array}
\right]  _{\tilde{n}},\ldots,\left[
\begin{array}
[c]{c}%
a_{2n-1}\\
a_{2n}%
\end{array}
\right]  _{\tilde{n}}\right]  =\left\{
\genfrac{}{}{0pt}{0}{\left[
\begin{array}
[c]{c}%
\mu^{\left(  m\right)  }\left[  a_{1},\ldots,a_{m}\right]  ,\\
\mu^{\left(  m\right)  }\left[  a_{m+1},\ldots,a_{2m}\right]
\end{array}
\right]  _{m},\ \ \ \ell_{\operatorname*{id}}=0,\ \ \tilde{n}=m,}{\left[
\begin{array}
[c]{c}%
\mu^{\left(  m\right)  }\left[  a_{1},\ldots,a_{m}\right]  ,\\
a_{m+1}%
\end{array}
\right]  _{m},\ \ \ \ell_{\operatorname*{id}}=1,\ \ \tilde{n}=\frac{m+1}{2},}%
\right.  , \label{ck-msn}%
\end{equation}
where $\ell_{\operatorname*{id}}=0,1$ is the number of \textit{intact
elements} in the r.h.s. of (\ref{ck-msn}).

\begin{definition}
The $n$-\textit{ary semigroup of the binary equivalence classes} is defined as%
\begin{equation}
\widetilde{\mathcal{S}}^{\left(  \tilde{n}\right)  }=\left\langle \left\{
\widetilde{\mathfrak{S}}\right\}  \mid\widetilde{\mathbf{\mu}}_{S}^{\left(
\tilde{n}\right)  }\right\rangle .
\end{equation}

\end{definition}

\begin{remark}
\label{ck-rem-nn}In general, for a fixed initial semigroup $\mathcal{S}%
^{\left(  m\right)  }$, the formulas of the doubles multiplication
(\ref{ck-mmnn}) and classes multiplication (\ref{ck-msn}) can be different (as
two choices of (\ref{ck-m2g})), because in one case we multiply concrete
elements, while in another case--their representatives. Moreover, the final
arities $n$ and $\tilde{n}$ need not coincide, if we use different choices in
r.h.s. of (\ref{ck-mmnn}) and (\ref{ck-msn}). For instance, compare the
$5$-ary (\ref{ck-m4}) and ternary (\ref{ck-m3}) multiplications of doubles
obtained from the same $5$-ary semigroup. This could lead to diverse possible
polyadic analogs of the group completions of the same semigroup. Nevertheless,
in what follows we will assume that%
\begin{equation}
\tilde{n}=n \label{ck-nn}%
\end{equation}
and choose the same product for elements and their representatives. Therefore,
all examples of polyadic direct products in \textbf{Subsection
\ref{ck-subsec-dp}} can be applied for the equivalence class multiplication
(\ref{ck-msn}) as well.
\end{remark}

The polyadic analog of the inverse is the querelement (\ref{ck-mgg})
\cite{dor3}, which does not need an identity at all. Therefore, the binary
formulas (\ref{ck-ma0})--(\ref{ck-g1}) do not work for $m\geq3$. So we have

\begin{assertion}
The presence of an identity is not necessary for constructing polyadic groups,
it can be optional and plays no role, and therefore we consider polyadic group
completion not of a polyadic monoid, but of a polyadic semigroup
$\mathcal{S}^{\left(  m\right)  }$.
\end{assertion}

Let us suppose that we can construct the querelement for any double
$\widetilde{\mathfrak{S}}$ in the $n$-ary semigroup $\widetilde{\mathcal{S}%
}^{\left(  n\right)  }$ of equivalence classes, which means that we define the
\textsf{unary} queroperation $\widetilde{\mathbf{\mu}}^{\left(  1\right)  }$
(see (\ref{ck-m1g}))%
\begin{equation}
\overline{\widetilde{\mathfrak{S}}}=\widetilde{\mathbf{\mu}}^{\left(
1\right)  }\left[  \widetilde{\mathfrak{S}}\right]  =\widetilde{\mathbf{\mu}%
}^{\left(  1\right)  }\left[  \left[
\begin{array}
[c]{c}%
a\\
b
\end{array}
\right]  _{n}\right]  =\left[
\begin{array}
[c]{c}%
\bar{a}\\
\bar{b}%
\end{array}
\right]  _{n}, \label{ck-ms0}%
\end{equation}
which satisfies (\ref{ck-mgg})%
\begin{equation}
\widetilde{\mathbf{\mu}}^{\left(  n\right)  }\left[  \widetilde{\mathfrak{S}%
}^{n-1},\overline{\widetilde{\mathfrak{S}}}\right]  =\widetilde{\mathfrak{S}}.
\label{ck-mss}%
\end{equation}

Then the $n$-ary semigroup $\widetilde{\mathcal{S}}^{\left(  n\right)  }$
becomes a polyadic group.

\begin{definition}
The $n$-\textit{ary group completion} of the $m$-ary semigroup $\mathcal{S}%
^{\left(  m\right)  }$ is%
\begin{equation}
\mathsf{K}_{0}^{\left(  m,n\right)  }\left(  \mathcal{S}^{\left(  m\right)
}\right)  =\left\langle \left\{  \widetilde{\mathfrak{S}}\right\}
\mid\widetilde{\mathbf{\mu}}^{\left(  n\right)  },\widetilde{\mathbf{\mu}%
}^{\left(  1\right)  }\right\rangle , \label{ck-kmn}%
\end{equation}
where $\widetilde{\mathbf{\mu}}^{\left(  n\right)  }$ and $\widetilde
{\mathbf{\mu}}^{\left(  1\right)  }$ are the $n$-ary multiplication of the
doubles (\ref{ck-msn}) and the queroperation (\ref{ck-ms0}).
\end{definition}

In this notation the standard binary case (\ref{ck-gs}) is%
\begin{equation}
\mathsf{K}_{0}\left(  \mathcal{S}\right)  =\mathsf{K}_{0}^{\left(  2,2\right)
}\left(  \mathcal{S}^{\left(  2\right)  }\right)  . \label{ck-k02}%
\end{equation}

\begin{remark}
The computation of the group completion $\mathsf{K}_{0}^{\left(  m,n\right)
}$ with abitrary $m$ and $n$ is possible, if we know both multiplications: of
the initial $m$-ary semigroup $\mu^{\left(  m\right)  }$ and of the $n$-ary
group of equivalent classes $\widetilde{\mathbf{\mu}}^{\left(  n\right)  }$
together with its queroperation $\widetilde{\mathbf{\mu}}^{\left(  1\right)
}$.
\end{remark}

\begin{remark}
As opposed to the binary case (\ref{ck-k02}), for a given $m$-ary semigroup
there can exist several associative products of doubles $\mathbf{\mu}%
^{\prime\left(  n\right)  }$ and corresponding products of classes
$\widetilde{\mathbf{\mu}}^{\left(  n\right)  }$ (see, e.g. (\ref{ck-mg1}) and
(\ref{ck-ncom}) or (\ref{ck-m4}) and (\ref{ck-m3})), and therefore the
polyadic group completion $\mathsf{K}_{0}^{\left(  m,n\right)  }$ is
\textsf{not unique}, in general. The number of different $\mathsf{K}%
_{0}^{\left(  m,n\right)  }$ coincides with the number of distinct associative
quivers in the doubles multiplication (\ref{ck-mmnn}) and class products
(\ref{ck-msn}) having the same arity shape \cite{dup2018a}.
\end{remark}

\begin{remark}
If we consider polyadic groups with \textsf{the same} arity shape of the
multiplication of classes $\widetilde{\mathbf{\mu}}^{\left(  n\right)  }$ and
\textsf{the same} queroperation $\widetilde{\mathbf{\mu}}^{\left(  1\right)
}$, then we can have a polyadic analog of universality (\ref{ck-fff}) for this
fixed completion group arity $n$.
\end{remark}

\begin{example}
Let us compute $\mathsf{K}_{0}^{\left(  m,n\right)  }$ for the $n$-ary
componentwise direct power (\ref{ck-mg1}) of the $m$-ary semigroup
$\mathcal{S}^{\left(  m\right)  }=\left\langle S\mid\mu^{\left(  m\right)
}\right\rangle $. Taking to account \textit{Remark} \textbf{\ref{ck-rem-nn}}
and (\ref{ck-nn}) we obtain for the equivalence classes doubles $\widetilde
{\mathfrak{S}}_{i}$ the following $m$-ary multiplication%
\begin{equation}
\widetilde{\mathbf{\mu}}_{compw}^{\left(  m\right)  }\left[  \widetilde
{\mathfrak{S}}_{1},\widetilde{\mathfrak{S}}_{2},\ldots,\widetilde
{\mathfrak{S}}_{m}\right]  =\left(
\begin{array}
[c]{c}%
\mu^{\left(  m\right)  }\left[  a_{1},a_{2},\ldots,a_{m}\right] \\[5pt]%
\mu^{\left(  m\right)  }\left[  b_{1},b_{2},\ldots,b_{m}\right]
\end{array}
\right)  ,\ \ \ \ \widetilde{\mathfrak{S}}_{i}=\left[
\begin{array}
[c]{c}%
a_{i}\\
b_{i}%
\end{array}
\right]  _{m},\ \ a_{i},b_{i}\in S,\ \label{ck-mcomp}%
\end{equation}
because $n=m$ for the componentwise multiplication (\ref{ck-mg1}). We resolve
the equation for the querelement (\ref{ck-mss}) as%
\begin{equation}
\overline{\widetilde{\mathfrak{S}}}_{compw}=\left[
\begin{array}
[c]{c}%
\bar{a}\\
\bar{b}%
\end{array}
\right]  _{m}\sim\left[
\begin{array}
[c]{c}%
\mu^{\left(  m\right)  }\left[  a^{\alpha},b^{m-\alpha}\right] \\
\mu^{\left(  m\right)  }\left[  a^{m-2+\alpha},b^{2-\alpha}\right]
\end{array}
\right]  _{m},
\end{equation}
where $\alpha=0,1,2$. Because all three solutions are in the same equivalence
class, we choose as the representative the symmetric choice $\alpha=1$%
\begin{equation}
\overline{\widetilde{\mathfrak{S}}}_{compw}=\widetilde{\mathbf{\mu}}%
_{compw}^{\left(  1\right)  }\left[  \widetilde{\mathfrak{S}}\right]  =\left[
\begin{array}
[c]{c}%
\bar{a}\\
\bar{b}%
\end{array}
\right]  _{m}=\left[
\begin{array}
[c]{c}%
\mu^{\left(  m\right)  }\left[  a,b^{m-1}\right] \\
\mu^{\left(  m\right)  }\left[  a^{m-1},b\right]
\end{array}
\right]  _{m}. \label{ck-sm1}%
\end{equation}
Therefore, for any $m$-ary semigroup with $m$-ary \textsf{componentwise}
direct power we obtain%
\begin{equation}
\mathsf{K}_{0}^{\left(  m,m\right)  }\left(  \mathcal{S}^{\left(  m\right)
}\right)  =\left\langle \left\{  \widetilde{\mathfrak{S}}\right\}
\mid\widetilde{\mathbf{\mu}}_{compw}^{\left(  m\right)  },\widetilde
{\mathbf{\mu}}_{compw}^{\left(  1\right)  }\right\rangle .
\end{equation}

\end{example}

\begin{example}
[\textsf{Negative numbers (continued)}]Let us introduce for the equivalence
classes $\widetilde{\mathfrak{S}}_{neg}=\left[
\begin{array}
[c]{c}%
-p\\
-q
\end{array}
\right]  _{3}$, $p,q\in\mathbb{N}$, $\gcd\left(  p,q\right)  =1$ (see
\textit{Example} \ref{ck-ex-neg}), their ternary \textsf{componentwise}
multiplication (\ref{ck-mcomp})%
\begin{equation}
\widetilde{\mathbf{\mu}}_{neg}^{\left(  3\right)  }\left[  \left[
\begin{array}
[c]{c}%
-p_{1}\\
-q_{1}%
\end{array}
\right]  _{3}\left[
\begin{array}
[c]{c}%
-p_{2}\\
-q_{2}%
\end{array}
\right]  _{3}\left[
\begin{array}
[c]{c}%
-p_{3}\\
-q_{3}%
\end{array}
\right]  _{3}\right]  =\left[
\begin{array}
[c]{c}%
-p_{1}p_{2}p_{3}\\
-q_{1}q_{2}q_{3}%
\end{array}
\right]  _{3},\ \ \ \ p_{i},q_{i}\in\mathbb{N}. \label{ck-mneg1}%
\end{equation}
We derive from (\ref{ck-mss}) the querelement%
\begin{equation}
\overline{\widetilde{\mathfrak{S}}}_{neg}=\widetilde{\mathbf{\mu}}%
_{neg}^{\left(  1\right)  }\left[  \widetilde{\mathfrak{S}}\right]
=\widetilde{\mathbf{\mu}}_{neg}^{\left(  1\right)  }\left[  \left[
\begin{array}
[c]{c}%
-p\\
-q
\end{array}
\right]  _{3}\right]  =\left[
\begin{array}
[c]{c}%
-pq^{2}\\
-p^{2}q
\end{array}
\right]  _{3},\ \ \ \ \forall p,q\in\mathbb{N},\ \ \gcd\left(  p,q\right)  =1.
\label{ck-qcomp}%
\end{equation}
Thus, we obtain the group completion of the negative numbers (with the
\textsf{componentwise} multiplication of classes $\widetilde{\mathbf{\mu}%
}_{neg}^{\left(  3\right)  }$)%
\begin{equation}
\mathsf{K}_{0}^{\left(  3,3\right)  }\left(  \mathbb{N}_{-}\right)
_{comp}=\left\langle \left\{  \widetilde{\mathfrak{S}}_{neg}\right\}
\mid\widetilde{\mathbf{\mu}}_{neg}^{\left(  3\right)  },\widetilde
{\mathbf{\mu}}_{neg}^{\left(  1\right)  }\right\rangle . \label{ck-kmin}%
\end{equation}
If we choose the ternary \textsf{noncomponentwise} multiplication
(\ref{ck-ncom}), then%
\begin{equation}
\widetilde{\mathbf{\mu}}_{neg2}^{\left(  3\right)  }\left[  \left[
\begin{array}
[c]{c}%
-p_{1}\\
-q_{1}%
\end{array}
\right]  _{3}\left[
\begin{array}
[c]{c}%
-p_{2}\\
-q_{2}%
\end{array}
\right]  _{3}\left[
\begin{array}
[c]{c}%
-p_{3}\\
-q_{3}%
\end{array}
\right]  _{3}\right]  =\left[
\begin{array}
[c]{c}%
-p_{1}q_{2}p_{3}\\
-q_{1}p_{2}q_{3}%
\end{array}
\right]  _{3},\ \ \ \ p_{i},q_{i}\in\mathbb{N}, \label{ck-ncomp}%
\end{equation}
and the querelement will be different from (\ref{ck-qcomp})%
\begin{equation}
\overline{\widetilde{\mathfrak{S}}}_{neg2}=\widetilde{\mathbf{\mu}}%
_{neg2}^{\left(  1\right)  }\left[  \widetilde{\mathfrak{S}}\right]
=\widetilde{\mathbf{\mu}}_{neg2}^{\left(  1\right)  }\left[  \left[
\begin{array}
[c]{c}%
-p\\
-q
\end{array}
\right]  _{3}\right]  =\left[
\begin{array}
[c]{c}%
-pq^{2}\\
-q^{3}%
\end{array}
\right]  _{3},\ \ \ \ \forall p,q\in\mathbb{N},\ \ \gcd\left(  p,q\right)  =1.
\end{equation}
Therefore, the second polyadic group completion of the negative numbers (with
the \textsf{noncomponentwise} multiplication of classes $\widetilde
{\mathbf{\mu}}_{neg2}^{\left(  3\right)  }$ (\ref{ck-ncomp}) being different
from $\widetilde{\mathbf{\mu}}_{neg}^{\left(  3\right)  }$ (\ref{ck-mneg1}))
becomes%
\begin{equation}
\mathsf{K}_{0}^{\left(  3,3\right)  }\left(  \mathbb{N}_{-}\right)
_{noncomp}=\left\langle \left\{  \widetilde{\mathfrak{S}}_{neg}\right\}
\mid\widetilde{\mathbf{\mu}}_{neg2}^{\left(  3\right)  },\widetilde
{\mathbf{\mu}}_{neg2}^{\left(  1\right)  }\right\rangle .
\end{equation}

\end{example}

Another simple example is the set of odd positive numbers.

\begin{example}
\label{ck-ex-nodd}Let $\mathbb{N}_{odd}=\left\{  2k+1\right\}  $,
$k\in\mathbb{N}_{0}$, then it is the ternary semigroup $\mathcal{S}%
_{odd}^{\left(  3\right)  }=\left\langle \mathbb{N}_{odd}\mid\mu
_{odd}^{\left(  3\right)  }\right\rangle $, where the multiplication is the
ordinary addition%
\begin{equation}
\mu_{odd}^{\left(  3\right)  }\left[  a,b,c\right]  =a+b+c,a,b,c\in
\mathbb{N}_{odd}. \label{ck-modd}%
\end{equation}
The doubles (\ref{ck-sab}) have the form%
\begin{equation}
\mathbf{S}_{odd}=\left(
\begin{array}
[c]{c}%
2k_{1}+1\\
2k_{2}+1
\end{array}
\right)  \in\mathbb{N}_{odd}\times\mathbb{N}_{odd},\ \ \ k_{i}\in
\mathbb{N}_{0}.
\end{equation}
The equivalence relations (\ref{ck-eq1}) and (\ref{ck-eq2}) become ($m=3$)%
\begin{align}
\left(
\begin{array}
[c]{c}%
2k_{1}+1\\
2k_{2}+1
\end{array}
\right)   &  \sim_{gauge}\left(
\begin{array}
[c]{c}%
2l_{1}+1\\
2l_{2}+1
\end{array}
\right) \nonumber\\
&  \Leftrightarrow\exists_{x,y\in\mathbb{N}_{odd}}\ \left(
\begin{array}
[c]{c}%
\mu_{odd}^{\left(  3\right)  }\left[  \left(  2k_{1}+1\right)  ^{2},x\right]
\\
\mu_{odd}^{\left(  3\right)  }\left[  \left(  2k_{2}+1\right)  ^{2},x\right]
\end{array}
\right)  =\left(
\begin{array}
[c]{c}%
\mu_{odd}^{\left(  3\right)  }\left[  \left(  2l_{1}+1\right)  ^{2},y\right]
\\
\mu_{odd}^{\left(  3\right)  }\left[  \left(  2l_{2}+1\right)  ^{2},y\right]
\end{array}
\right)  ,\ \ \ k_{i},l_{i}\in\mathbb{N}_{0},
\end{align}
and%
\begin{align}
\left(
\begin{array}
[c]{c}%
2k_{1}+1\\
2k_{2}+1
\end{array}
\right)   &  \sim_{twist}\left(
\begin{array}
[c]{c}%
2l_{1}+1\\
2l_{2}+1
\end{array}
\right) \nonumber\\
&  \Leftrightarrow\exists_{x,y\in\mathbb{N}_{odd}}\ \left(  \mu_{odd}^{\left(
3\right)  }\right)  ^{\circ2}\left[  \left(  2k_{1}+1\right)  ^{2},\left(
2l_{2}+1\right)  ^{2},z\right]  =\left(  \mu_{odd}^{\left(  3\right)
}\right)  ^{\circ2}\left[  \left(  2k_{2}+1\right)  ^{2},\left(
2l_{1}+1\right)  ^{2},z\right]  .
\end{align}

For instance,%
\begin{equation}%
\begin{array}
[c]{cc}%
\begin{array}
[c]{c}%
\left(
\begin{array}
[c]{c}%
3\\
1
\end{array}
\right)  \sim\left(
\begin{array}
[c]{c}%
5\\
3
\end{array}
\right)  \sim\left(
\begin{array}
[c]{c}%
7\\
5
\end{array}
\right)  \ldots,\\[10pt]%
\left(
\begin{array}
[c]{c}%
5\\
1
\end{array}
\right)  \sim\left(
\begin{array}
[c]{c}%
7\\
3
\end{array}
\right)  \sim\left(
\begin{array}
[c]{c}%
9\\
5
\end{array}
\right)  \ldots,\\[10pt]%
\left(
\begin{array}
[c]{c}%
7\\
1
\end{array}
\right)  \sim\left(
\begin{array}
[c]{c}%
9\\
3
\end{array}
\right)  \sim\left(
\begin{array}
[c]{c}%
11\\
5
\end{array}
\right)  \ldots,
\end{array}
&
\begin{array}
[c]{c}%
\left(
\begin{array}
[c]{c}%
1\\
3
\end{array}
\right)  \sim\left(
\begin{array}
[c]{c}%
3\\
5
\end{array}
\right)  \sim\left(
\begin{array}
[c]{c}%
5\\
7
\end{array}
\right)  \ldots,\\[10pt]%
\left(
\begin{array}
[c]{c}%
1\\
5
\end{array}
\right)  \sim\left(
\begin{array}
[c]{c}%
3\\
7
\end{array}
\right)  \sim\left(
\begin{array}
[c]{c}%
5\\
9
\end{array}
\right)  \ldots,\\[10pt]%
\left(
\begin{array}
[c]{c}%
1\\
7
\end{array}
\right)  \sim\left(
\begin{array}
[c]{c}%
3\\
9
\end{array}
\right)  \sim\left(
\begin{array}
[c]{c}%
5\\
11
\end{array}
\right)  \ldots.
\end{array}
\end{array}
\end{equation}

In general, two elements
\begin{equation}
\left(
\begin{array}
[c]{c}%
a\\
b
\end{array}
\right)  ,\left(
\begin{array}
[c]{c}%
a+2\\
b+2
\end{array}
\right)  ,\ \ \ \ a,b\in\mathbb{N}_{odd}, \label{ck-kp1}%
\end{equation}
are in the same equivalence class $\left[
\begin{array}
[c]{c}%
a\\
b
\end{array}
\right]  \in\mathbb{N}_{odd}\times\mathbb{N}_{odd}\diagup\sim$. The minimal
representatives are the doubles%
\begin{equation}
\widetilde{\mathfrak{S}}_{up}=\left[
\begin{array}
[c]{c}%
2k+1\\
1
\end{array}
\right]  ,\ \ \ \widetilde{\mathfrak{S}}_{down}=\left[
\begin{array}
[c]{c}%
1\\
2k+1
\end{array}
\right]  ,\ \ \ \ \ k\in\mathbb{N}. \label{ck-k12}%
\end{equation}
Thus, the structure of the classes is the following%
\begin{equation}
\left\{  \widetilde{\mathfrak{S}}\right\}  =\left\{  \widetilde{\mathfrak{S}%
}_{up}\right\}  \cup\left\{  \widetilde{\mathfrak{S}}_{down}\right\}
,\ \ \ \ \left\{  \widetilde{\mathfrak{S}}_{up}\right\}  \cap\left\{
\widetilde{\mathfrak{S}}_{down}\right\}  =\varnothing.
\end{equation}
We do not have an identity in the ternary semigroup $\mathcal{S}%
_{odd}^{\left(  3\right)  }$, and therefore the representatives (\ref{ck-k12})
cannot be mapped to negative and positive numbers, as in the binary case
(\ref{ck-nmz}), and moreover their multiplication is ternary (\ref{ck-modd}).

To find the polyadic group structure (i.e. querelements (\ref{ck-mgg})), for
the ternary multiplication of the classes $\widetilde{\mathfrak{S}}$ we
consider two cases:

\begin{enumerate}
\item The componentwise multiplication (\ref{ck-mg1}). We use the ternary
multiplication of the classes as ordinary addition%
\begin{align}
&  \widetilde{\mathbf{\mu}}_{odd1}^{\left(  3\right)  }\left[  \left[
\begin{array}
[c]{c}%
2k_{1}+1\\
1
\end{array}
\right]  ,\left[
\begin{array}
[c]{c}%
2k_{2}+1\\
1
\end{array}
\right]  ,\left[
\begin{array}
[c]{c}%
2k_{3}+1\\
1
\end{array}
\right]  \right] \nonumber\\
&  =\left[
\begin{array}
[c]{c}%
2\left(  k_{1}+k_{2}+k_{3}\right)  +3\\
3
\end{array}
\right]  \sim\left[
\begin{array}
[c]{c}%
2\left(  k_{1}+k_{2}+k_{3}\right)  +1\\
1
\end{array}
\right]  ,\ \ \ \ k_{i}\in\mathbb{N},
\end{align}
and%
\begin{align}
&  \widetilde{\mathbf{\mu}}_{odd1}^{\left(  3\right)  }\left[  \left[
\begin{array}
[c]{c}%
1\\
2k_{1}+1
\end{array}
\right]  ,\left[
\begin{array}
[c]{c}%
1\\
2k_{2}+1
\end{array}
\right]  ,\left[
\begin{array}
[c]{c}%
1\\
2k_{3}+1
\end{array}
\right]  \right] \nonumber\\
&  =\left[
\begin{array}
[c]{c}%
3\\
2\left(  k_{1}+k_{2}+k_{3}\right)  +3
\end{array}
\right]  \sim\left[
\begin{array}
[c]{c}%
1\\
2\left(  k_{1}+k_{2}+k_{3}\right)  +1
\end{array}
\right]  ,\ \ \ \ k_{i}\in\mathbb{N}.
\end{align}
Then, apply the general formula for the querelement (\ref{ck-sm1}) to our
ternary case%
\begin{equation}
\overline{\widetilde{\mathfrak{S}}}_{compw}=\left[
\begin{array}
[c]{c}%
\bar{a}\\
\bar{b}%
\end{array}
\right]  _{3}=\left[
\begin{array}
[c]{c}%
\mu^{\left(  3\right)  }\left[  a,b^{2}\right] \\
\mu^{\left(  3\right)  }\left[  a^{2},b\right]
\end{array}
\right]  _{3}, \label{ck-q3}%
\end{equation}
which gives using the equivalence relations (\ref{ck-kp1})%
\begin{align}
\overline{\widetilde{\mathfrak{S}}}_{up}  &  =\widetilde{\mathbf{\mu}}%
_{odd1}^{\left(  1\right)  }\left[  \widetilde{\mathfrak{S}}_{up}\right]
=\overline{\left[
\begin{array}
[c]{c}%
2k+1\\
1
\end{array}
\right]  }\sim\left[
\begin{array}
[c]{c}%
1\\
2k+1
\end{array}
\right]  =\widetilde{\mathfrak{S}}_{down},\label{ck-u3}\\
\overline{\widetilde{\mathfrak{S}}}_{down}  &  =\widetilde{\mathbf{\mu}%
}_{odd1}^{\left(  1\right)  }\left[  \widetilde{\mathfrak{S}}_{down}\right]
=\overline{\left[
\begin{array}
[c]{c}%
1\\
2k+1
\end{array}
\right]  }\sim\left[
\begin{array}
[c]{c}%
2k+1\\
1
\end{array}
\right]  =\widetilde{\mathfrak{S}}_{up},\ \ \ \ \ k\in\mathbb{N}.
\label{ck-d3}%
\end{align}

Thus, for the \textsf{componentwise} ternary group completion of the ternary
semigroup of odd numbers $\mathbb{N}_{odd}$ we obtain%
\begin{equation}
\mathsf{K}_{0}^{\left(  3,3\right)  }\left(  \mathbb{N}_{odd}\right)
_{compw}=\left\langle \left\{  \widetilde{\mathfrak{S}}_{up}\right\}
\cup\left\{  \widetilde{\mathfrak{S}}_{down}\right\}  \mid\widetilde
{\mathbf{\mu}}_{odd1}^{\left(  3\right)  },\widetilde{\mathbf{\mu}}%
_{odd1}^{\left(  1\right)  }\right\rangle .
\end{equation}

\item The noncomponentwise Post-like multiplication (\ref{ck-ncom}). In this
case, we denote the multiplication of classes corresponding to the double
product (\ref{ck-ncom}) by $\widetilde{\mathbf{\mu}}_{odd2}^{\left(  3\right)
}$, then the general formula for the querelement will be different from
(\ref{ck-q3}) and have the form%
\begin{equation}
\overline{\widetilde{\mathfrak{S}}}_{ncompw}=\left[
\begin{array}
[c]{c}%
\bar{a}\\
\bar{b}%
\end{array}
\right]  _{3}=\left[
\begin{array}
[c]{c}%
\mu^{\left(  3\right)  }\left[  a^{2},b\right] \\
\mu^{\left(  3\right)  }\left[  a,b^{2}\right]
\end{array}
\right]  _{3},
\end{equation}
which gives (using the equivalence relations (\ref{ck-kp1}))%
\begin{align}
\overline{\widetilde{\mathfrak{S}}}_{up}  &  =\widetilde{\mathbf{\mu}}%
_{odd2}^{\left(  1\right)  }\left[  \widetilde{\mathfrak{S}}_{up}\right]
=\overline{\left[
\begin{array}
[c]{c}%
2k+1\\
1
\end{array}
\right]  }\sim\left[
\begin{array}
[c]{c}%
2k+1\\
1
\end{array}
\right]  =\widetilde{\mathfrak{S}}_{up},\\
\overline{\widetilde{\mathfrak{S}}}_{down}  &  =\widetilde{\mathbf{\mu}%
}_{odd2}^{\left(  1\right)  }\left[  \widetilde{\mathfrak{S}}_{down}\right]
=\overline{\left[
\begin{array}
[c]{c}%
1\\
2k+1
\end{array}
\right]  }\sim\left[
\begin{array}
[c]{c}%
1\\
2k+1
\end{array}
\right]  =\widetilde{\mathfrak{S}}_{down},\ \ \ \ \ k\in\mathbb{N}.
\end{align}

Now, each element is a kind of polyadic reflection obeying $\overline
{\widetilde{\mathfrak{S}}}=\widetilde{\mathfrak{S}}$ (an analog of $g^{-1}=g$
in binary groups). Thus, the \textsf{noncomponentwise} ternary group
completion of the ternary semigroup of odd numbers $\mathbb{N}_{odd}$ is%
\begin{equation}
\mathsf{K}_{0}^{\left(  3,3\right)  }\left(  \mathbb{N}_{odd}\right)
_{ncompw}=\left\langle \left\{  \widetilde{\mathfrak{S}}_{up}\right\}
\cup\left\{  \widetilde{\mathfrak{S}}_{down}\right\}  \mid\widetilde
{\mathbf{\mu}}_{odd2}^{\left(  3\right)  },\widetilde{\mathbf{\mu}}%
_{odd2}^{\left(  1\right)  }\right\rangle .
\end{equation}

\end{enumerate}
\end{example}

\begin{example}
Consider the following $4$-ary semigroup of complex matrices%
\begin{align}
\mathcal{S}_{matr}^{\left(  4\right)  }  &  =\left\langle S_{matr}\mid
\mu_{matr}^{\left(  4\right)  }\right\rangle ,\ \ S_{matr}=\left\{  a\right\}
,\ \ a=\left(
\begin{array}
[c]{cc}%
u & 0\\
0 & 0
\end{array}
\right)  ,\ \ u\in\mathbb{C},\label{ck-m4a}\\
\mu_{matr}^{\left(  4\right)  }\left[  a_{1},a_{2},a_{3},a_{4}\right]   &
=a_{1}+\varepsilon a_{2}+\varepsilon^{2}a_{3}+a_{4}, \label{ck-m4b}%
\end{align}
where $\varepsilon=e^{\frac{2\pi}{3}}$, $\varepsilon^{3}=1$.

We construct the doubles $\mathbf{S=}\left(
\begin{array}
[c]{c}%
a\\
b
\end{array}
\right)  \in S_{matr}\times S_{matr}$ with the componentwise multiplication
(\ref{ck-mg1}) which does not change the arity. Then we write the
\textquotedblleft twist-like\textquotedblright\ equivalence relation
(\ref{ck-eq2}) for our case%
\begin{equation}
\left(
\begin{array}
[c]{c}%
a_{1}\\
b_{1}%
\end{array}
\right)  \sim_{twist}\left(
\begin{array}
[c]{c}%
a_{2}\\
b_{2}%
\end{array}
\right)  \Longleftrightarrow\exists_{z\in S}\ \left(  \mu^{\left(  4\right)
}\right)  ^{\circ2}\left[  a_{1}^{3},b_{2}^{3},z\right]  =\left(  \mu^{\left(
4\right)  }\right)  ^{\circ2}\left[  a_{2}^{3},b_{1}^{3},z\right]
,\ \ a_{i},b_{i},z\in S_{matr}.
\end{equation}

It follows from (\ref{ck-m4a})--(\ref{ck-m4b}), that $\mathcal{S}%
_{matr}^{\left(  4\right)  }$ is an idempotent semigroup, because
$\mu^{\left(  4\right)  }\left[  a^{4}\right]  =a$, $\forall a\in S_{matr}$,
moreover, $\mu^{\left(  4\right)  }\left[  a^{3},b\right]  =b$, $\forall
a,b\in S_{matr}$, and so $\left(  \mu^{\left(  4\right)  }\right)  ^{\circ
2}\left[  a_{1}^{3},b_{2}^{3},z\right]  =\left(  \mu^{\left(  m\right)
}\right)  ^{\circ2}\left[  a_{2}^{3},b_{1}^{3},z\right]  =z$ identically for
$\forall a_{i},b_{i},z\in S_{matr}$. Thus, we do not have different
equivalence classes at all, and therefore for (\ref{ck-m4a})%
\begin{equation}
\mathsf{K}_{0}^{\left(  4,4\right)  }\left(  \mathcal{S}_{matr}^{\left(
4\right)  }\right)  =1,
\end{equation}
in our multiplicative notation.
\end{example}

Let us consider the case, when the completion group of classes has different
arity from that of the initial semigroup, i.e $\mathsf{K}_{0}^{\left(
m,n\right)  }\left(  \mathcal{S}^{\left(  m\right)  }\right)  $ with $n\neq m$.

\begin{example}
The representatives of the residue (congruence) class $S_{res}=\left[  \left[
7\right]  \right]  _{10}=\left\{  10k+7\right\}  $, $k\in\mathbb{N}_{0}$ form
a $5$-ary semigroup with respect to multiplication (see the arity shape for
residue classes $\left[  \left[  a\right]  \right]  _{b}$ in \cite{dup2017a}).
We take%
\begin{equation}
\mathcal{S}_{res}^{\left(  5\right)  }=\left\langle \left\{  10k+7\right\}
\mid\mu_{res}^{\left(  5\right)  }\right\rangle , \label{ck-s5}%
\end{equation}
where $5$-ary multiplication $\mu_{res}^{\left(  5\right)  }$ is the ordinary
product (in $\mathbb{Z}_{+}$) $\mu_{res}^{\left(  5\right)  }\left[
a_{1},a_{2},a_{3},a_{4},a_{5}\right]  =a_{1}a_{2}a_{3}a_{4}a_{5}$,
$a_{i}=10k_{i}+7$, $k_{i}\in\mathbb{N}_{0}$.

Now the doubles (\ref{ck-sab}) are%
\begin{equation}
\mathbf{S}_{res}=\left(
\begin{array}
[c]{c}%
10k+7\\
10l+7
\end{array}
\right)  \in\left[  \left[  7\right]  \right]  _{10}\times\left[  \left[
7\right]  \right]  _{10},\ \ \ k,l\in\mathbb{N}_{0}. \label{ck-res}%
\end{equation}
The equivalence relations (\ref{ck-eq1}) and (\ref{ck-eq2}) for $m=5$ are%
\begin{align}
\left(
\begin{array}
[c]{c}%
10k_{1}+7\\
10l_{1}+7
\end{array}
\right)   &  \sim_{gauge}\left(
\begin{array}
[c]{c}%
10k_{2}+7\\
10l_{2}+7
\end{array}
\right) \nonumber\\
&  \Leftrightarrow\exists_{x,y\in S_{res}}\ \left(
\begin{array}
[c]{c}%
\left(  10k_{1}+7\right)  ^{4}x\\
\left(  10l_{1}+7\right)  ^{4}x
\end{array}
\right)  =\left(
\begin{array}
[c]{c}%
\left(  10k_{2}+7\right)  ^{4}y\\
\left(  10l_{2}+7\right)  ^{4}y
\end{array}
\right)  ,\ \ \ k_{i},l_{i}\in\mathbb{N}_{0}, \label{ck-ga}%
\end{align}
and%
\begin{align}
\left(
\begin{array}
[c]{c}%
10k_{1}+7\\
10l_{1}+7
\end{array}
\right)   &  \sim_{twist}\left(
\begin{array}
[c]{c}%
10k_{2}+7\\
10l_{2}+7
\end{array}
\right) \nonumber\\
&  \Leftrightarrow\exists_{x,y\in S_{res}}\ \left(  10k_{1}+7\right)
^{4}\left(  10l_{2}+7\right)  ^{4}z=\left(  10k_{2}+7\right)  ^{4}\left(
10l_{1}+7\right)  ^{2}z. \label{ck-kl1}%
\end{align}

It follows from the \textquotedblleft gauge\textquotedblright\ shifts
(\ref{ck-ga}) that the components of the double (\ref{ck-res}) are mutually
prime. Some of the equivalence relations are (we present only those with
$k_{1}<k_{2}$)%
\begin{equation}%
\begin{array}
[c]{cc}%
\begin{array}
[c]{c}%
\left(
\begin{array}
[c]{c}%
7\\
17
\end{array}
\right)  \sim\left(
\begin{array}
[c]{c}%
77\\
187
\end{array}
\right)  \sim\left(
\begin{array}
[c]{c}%
147\\
357
\end{array}
\right)  \ldots,\\[10pt]%
\left(
\begin{array}
[c]{c}%
7\\
77
\end{array}
\right)  \sim\left(
\begin{array}
[c]{c}%
17\\
187
\end{array}
\right)  \sim\left(
\begin{array}
[c]{c}%
27\\
297
\end{array}
\right)  \ldots,\\[10pt]%
\left(
\begin{array}
[c]{c}%
17\\
47
\end{array}
\right)  \sim\left(
\begin{array}
[c]{c}%
187\\
517
\end{array}
\right)  \sim\left(
\begin{array}
[c]{c}%
357\\
987
\end{array}
\right)  \ldots,
\end{array}
&
\begin{array}
[c]{c}%
\left(
\begin{array}
[c]{c}%
27\\
37
\end{array}
\right)  \sim\left(
\begin{array}
[c]{c}%
297\\
407
\end{array}
\right)  \sim\left(
\begin{array}
[c]{c}%
567\\
777
\end{array}
\right)  \ldots,\\[10pt]%
\left(
\begin{array}
[c]{c}%
57\\
87
\end{array}
\right)  \sim\left(
\begin{array}
[c]{c}%
247\\
377
\end{array}
\right)  \sim\left(
\begin{array}
[c]{c}%
437\\
667
\end{array}
\right)  \ldots,\\[10pt]%
\left(
\begin{array}
[c]{c}%
87\\
377
\end{array}
\right)  \sim\left(
\begin{array}
[c]{c}%
147\\
637
\end{array}
\right)  \sim\left(
\begin{array}
[c]{c}%
177\\
767
\end{array}
\right)  \ldots.
\end{array}
\end{array}
\end{equation}

We denote the equivalence class by%
\begin{equation}
\widetilde{\mathfrak{S}}=\widetilde{\mathfrak{S}}_{res}=\left[
\begin{array}
[c]{c}%
a\\
b
\end{array}
\right]  _{res}=\left[
\begin{array}
[c]{c}%
10k+7\\
10l+7
\end{array}
\right]  _{res}\in S_{res}\times S_{res}\diagup\sim,\ \ k,l\in\mathbb{N}_{0},
\label{ck-sres}%
\end{equation}
and we use the minimal representative, when needed and possible. The polyadic
group of the equivalence classes can be obtained, if we define the classes
multiplication and the querelement. If we choose the $5$-ary componentwise
multiplication (\ref{ck-mcomp}), then we obtain the group completion of the
same arity $\mathsf{K}_{0}^{\left(  5,5\right)  }\left(  \mathcal{S}%
_{res}^{\left(  5\right)  }\right)  $, as in the previous \textit{Example
}\ref{ck-ex-nodd}. The more exotic case is to use for the classes doubles
$\widetilde{\mathfrak{S}}_{res}$ the changing arity product (\ref{ck-m3}),
which gives us the ternary group of equivalence classes, which is built from
the $5$-ary semigroup. Indeed, let $\widetilde{\mathbf{\mu}}_{res}^{\left(
3\right)  }$ be the ternary multiplication of the equivalence classes defined
by (cf. (\ref{ck-m3}))%
\begin{equation}
\widetilde{\mathbf{\mu}}_{res}^{\left(  3\right)  }\left[  \widetilde
{\mathfrak{S}}_{1},\widetilde{\mathfrak{S}}_{2},\widetilde{\mathfrak{S}}%
_{3}\right]  =\left[
\begin{array}
[c]{c}%
\left(  10k_{1}+7\right)  \left(  10k_{2}+7\right)  \left(  10k_{3}+7\right)
\left(  10l_{1}+7\right)  \left(  10l_{2}+7\right) \\
10l_{3}+7
\end{array}
\right]  _{res},\ \ k_{i},l_{i}\in\mathbb{N}_{0}.
\end{equation}

The querelement $\overline{\widetilde{\mathfrak{S}}}$ can be determined from
the manifest form of the ternary multiplication $\widetilde{\mathbf{\mu}%
}_{res}^{\left(  3\right)  }$ of the equivalence classes, and it has the
general form%
\begin{equation}
\overline{\widetilde{\mathfrak{S}}}=\left[
\begin{array}
[c]{c}%
\bar{a}\\
\bar{b}%
\end{array}
\right]  _{res}\equiv\widetilde{\mathbf{\mu}}_{res}^{\left(  1\right)
}\left[  \widetilde{\mathfrak{S}}\right]  \sim\left[
\begin{array}
[c]{c}%
a^{3}b^{2}\\
a^{2}b^{3}%
\end{array}
\right]  _{res}\sim\left[
\begin{array}
[c]{c}%
\left(  10k+7\right)  ^{3}\left(  10l+7\right)  ^{2}\\
\left(  10k+7\right)  ^{2}\left(  10l+7\right)  ^{3}%
\end{array}
\right]  _{res},\ a,b\in S_{res},\ \ k_{i},l_{i}\in\mathbb{N}.
\end{equation}

Thus, the ternary group of the classes (\ref{ck-sres})%
\begin{equation}
\mathsf{K}_{0}^{\left(  5,3\right)  }\left(  \mathcal{S}_{res}^{\left(
5\right)  }\right)  =\left\langle \left\{  \widetilde{\mathfrak{S}}%
_{res}\right\}  \mid\widetilde{\mathbf{\mu}}_{res}^{\left(  3\right)
},\widetilde{\mathbf{\mu}}_{res}^{\left(  1\right)  }\right\rangle
\end{equation}
is the completion group of the $5$-ary semigroup $\mathcal{S}_{res}^{\left(
5\right)  }$ (\ref{ck-s5}).
\end{example}

\bigskip

\textbf{Acknowledgements.} The author is deeply thankful to Joachim Cuntz,
Siegfried Echterhoff and Christian Voigt for long-ago explanations of
$K$-theory which indeed now re-awakened my former interest to consider higher
arities in this promising direction, and grateful to Vladimir Akulov, Mike
Hewitt, Dimitrij Leites, Thomas Nordahl, Vladimir Tkach, Raimund Vogl and
Alexander Voronov for useful discussions, and valuable help.


\pagestyle{emptyf}

\end{document}